\documentclass[12pt]{article}

\usepackage{amsmath,amsthm,amssymb,amsfonts,amscd, mathrsfs, mathtools, dsfont}
\usepackage[hyperindex,breaklinks]{hyperref}
\usepackage{cleveref}
\usepackage{tikz}
\usepackage{tikz-cd}
\usepackage{cancel}
\usepackage{fullpage}
\usepackage{enumerate}
\usepackage[shortlabels]{enumitem}
\usepackage{makecell}
\usetikzlibrary{matrix,arrows,decorations.pathmorphing}

\theoremstyle{definition}
\newtheorem{theorem}{Theorem}[section]
\newtheorem{corollary}[theorem]{Corollary}
\newtheorem{lemma}[theorem]{Lemma}
\newtheorem{proposition}[theorem]{Proposition}
\newtheorem{conjecture}[theorem]{Conjecture} 
 
\newtheorem{definition}[theorem]{Definition}

\numberwithin{equation}{section}
\numberwithin{theorem}{section}

\definecolor{darkgreen}{rgb}{0,0.5,0}
\definecolor{darkblue}{rgb}{0,0.1,0.5}

\hypersetup{
    colorlinks=true, % make the links colored
    linkcolor=darkblue, % color TOC links in blue
    urlcolor=blue, % color URLs in red
    linktoc=all, % 'all' will create links for everything in the TOC
    citecolor=darkgreen
    }

\newcommand{\Id}{\mathrm{Id}}

\newcommand{\Hom}{\mathrm{Hom}}

\newcommand{\mfB}{\mathfrak{B}}

\newcommand{\mfg}{\mathfrak{g}}
\newcommand{\mfh}{\mathfrak{h}}

\newcommand{\mcA}{\mathcal{A}}
\newcommand{\mcB}{\mathcal{B}}
\newcommand{\mcC}{\mathcal{C}}

\newcommand{\mcF}{\mathcal{F}}
\newcommand{\mcG}{\mathcal{G}}
\newcommand{\mcH}{\mathcal{H}}

\newcommand{\mcL}{\mathcal{L}}

\newcommand{\mcW}{\mathcal{W}}

\newcommand{\mbbC}{\mathbb{C}}

\newcommand{\mbbZ}{\mathbb{Z}}

\usepackage{parskip}
\usepackage{amssymb} 
\usepackage{amsbsy}
\usepackage{amscd}
\usepackage{amsmath}
\usepackage{amsthm, lipsum}
\newcommand{\C}{\mathcal{C}}
\newcommand\doi[2]{\href{http://dx.doi.org/#1}{#2}}
\newcommand{\U}{\overline{U}_q^H(\mfg)}

\newcommand{\irred}[1]{S^{#1}}

\makeatletter
\def\thm@space@setup{%
  \thm@preskip=3mm %plus 1cm minus 2cm
  \thm@postskip=\thm@preskip % or whatever, if you don't want them to be equal
}
\makeatother

\begin{document}

\title{Uprolling Unrolled Quantum Groups}
\author{Thomas Creutzig and Matthew Rupert}
\maketitle

\begin{abstract}

We construct families of commutative (super) algebra objects in the category of weight modules for the unrolled restricted quantum group $\overline{U}_q^H(\mfg)$ of a simple Lie algebra $\mfg$ at roots of unity, and study their categories of local modules. We determine their simple modules and derive conditions for these categories being finite, non-degenerate, and ribbon. Motivated by numerous examples in the $\mfg=\mathfrak{sl}_2$ case, we expect some of these categories to compare nicely to categories of modules for vertex operator algebras. We focus in particular on examples expected to correspond to the higher rank triplet vertex algebra $W_Q(r)$ of Feigin and Tipunin \cite{FT} and the $B_Q(r)$ algebras of \cite{C1}.

\end{abstract}

%\tableofcontents
\section{Introduction}

The unrolled quantum groups $U_q^H(\mfg)$ associated to a finite dimensional simple complex Lie algebra $\mfg$ were initially introduced and studied, primarily at odd roots of unity, as examples for producing link invariants through their categories of weight modules (see \cite{GP1,GP2,GPT}). 
A different perspective comes from the conjectural relation to representation categories of vertex operator algebras, which can be viewed as rich and challenging variants of the famous Kazhdan-Lusztig correspondence, i.e. the equivalence of rigid tensor categories of quantum groups and affine vertex algebras at generic $q$ and generic level respectively \cite{KL1, KL2, KL3, KL4}.
These
connections to vertex operator algebras appeared first through the proposed relation of the restricted quantum group of $\mathfrak{sl}_2$ at even root of unity 
and triplet vertex algebras \cite{FGST1, FGST2, FGST3, FGST4} pioneered by Feigin, Gainutdinov, Semikhatov and Tipunin. Cases beyond $\mathfrak{sl}_2$ have been studied by Lentner and Flandoli \cite{FL, L}. Proofs of these conjectures seem to be currently out of reach and so far the case of $\mathfrak{sl}_2$ at $q=\sqrt{-1}$ is best understood \cite{Run, GR}. 

Even for $\mathfrak{sl}_2$ establishing a braided tensor category structure on the category of weight modules of the restricted quantum group at even root $q= e^{\pi i/r}$ is difficult since the standard tensor product is not braidable \cite{KS}.
This issue has only been solved recently, i.e. it is the representation category for a factorizable quasi Hopf algebra whose underlying algebra is that of the restricted quantum group of $\mathfrak{sl}_2$ \cite{CGR}. The proof idea is due to vertex algebra observations. Namely the triplet vertex algebra $W_{A_1}(r)$  is a simple current extension\footnote{Simple current is the vertex algebra terminology for an invertible object.} of the singlet vertex algebra $W^0_{A_1}(r)$ . Since the category of weight modules of the unrolled quantum group is supposed to be equivalent to a category of modules of the singlet algebra this suggests that one can perform a simple current extension in the category of weight modules of the unrolled quantum group and the category of local modules of this extension gives us a rigid braided tensor category whose underlying algebra is the restricted quantum group. We like to call this procedure uprolling. This gives an example where quantum groups have benefitted from the connection to vertex algebras. Conversely vertex algebras benefit immensely from this relation and there are numerous further connections between the unrolled quantum groups and vertex operator algebras which have been established in the $\mfg=\mathfrak{sl}_2$ case at even roots of unity (see \cite{CMR,CGR,ACKR}) via the construction of braided tensor categories which compare nicely to certain representation categories of the corresponding vertex operator algebras. These categories were constructed as the categories of local modules for commutative algebra objects, and it is expected that these connections should extend to the higher rank cases as well. 

The category of weight modules $\mathrm{Rep}\,\overline{U}_q^H(\mfg)^{wt}$ (see Subsection \ref{subunrolled}) for the unrolled restricted quantum group was studied at arbitrary roots of unity in \cite{R}. In this article, we classify commutative algebra objects and supercommutative superalgebra objects built from simple currents (see Subsection \ref{subcurrents}) in $\mathrm{Rep}\,\overline{U}_q^H(\mfg)^{wt}$ and study their categories of local modules. This yields a wealth of examples of non-degenerate ribbon categories of both finite and non-finite type. Finite tensor categories with similar ideas have been constructed in \cite{N, GLO}. We now present our results and then explain the motivation from and application for vertex algebras and physics.

\subsection{Results}

Let $\mathcal{B}$ be a braided tensor category whose objects carry the structure of a vector space. Following conformal field theory language a simple object in $\mathcal B$ is called a simple current if it is invertible with respect to the tensor product. The tensor product of two simple currents is a simple current and so the Grothendieck ring of the pointed subcategory of simple currents of $\mathcal B$ is the group algebra of an abelian group. We assume this group to be free and associate to it a lattice $\mathcal L$ with quadratic form $\langle - , - \rangle : \mathcal L \times \mathcal L \rightarrow \mathbb Q/\ell \mathbb Z$ for some fixed integer $\ell$.  
The simple current associated to $\lambda$ in $\mathcal L$ is denoted by $\mbbC_\lambda$ and we assume this quadratic form determines the braiding of simple currents, e.g. 
 $c_{\mbbC_{\lambda},\mbbC_{\mu}}(v_{\lambda} \otimes v_{\mu})=q^{\langle \lambda, \mu \rangle}v_{\mu} \otimes v_{\lambda}$ for a primitive $\ell$-th root of unity q. For any lattice $L \subset \mathcal{L}$, define the object 
\[ \mathcal{A}_L:=\bigoplus\limits_{\lambda \in L} \mbbC_{\lambda} \in \mcB^{\oplus}.\]
Then we have the following classification result:
\begin{theorem}
$\mcA_L$ is an associative algebra object for all $L \subset \mathcal{L}$. $\mcA_L$ is commutative if and only if $\sqrt{2/\ell}\, L$ is an even lattice, that is if and only if $\langle \lambda,\lambda \rangle \in \ell \mbbZ$ and $2\langle \lambda,\mu \rangle \in \ell \mbbZ$ 
for all $\lambda,\mu \in L$.
\end{theorem}
Let $L \subset L^{\mu} \subset \mcL$ be the lattice generated by $\mu \in \mcL$ and $L$. We can then define the associated algebra object $\mcA_{L^{\mu}}:= \bigoplus_{\lambda \in L^{\mu}} \mbbC_{\lambda}$, and we have the following proposition (Proposition \ref{superalg} and Corollary \ref{superalgclass}):
\begin{proposition}
Let $L \subset \mcL$ such that $\mcA_L$ is commutative and $\mu \in \mcL$ such that $\mu \not \in L$ and $2 \mu \in L$, then $\mcA_{L^{\mu}}$ is a superalgebra. $\mcA_{L^{\mu}}$ is supercommutative if and only if
\begin{align*}
2\langle \mu, \mu \rangle \in \ell \mbbZ \setminus 2\ell \mbbZ  \qquad \text{and} \qquad  2\langle \mu, \lambda \rangle  \in \ell \mbbZ 
\end{align*}
for all $\lambda\in L$, and all supercommutative superalgebras $A=A^{\bar{0}} \oplus A^{\bar{1}}$ which are direct sums of simple currents such that $A^{\bar{0}}=\mcA_L$ for some lattice $L \subset \mcL$ and $A^{\bar{1}}$ is a non-trivial simple object in $\mathrm{Rep}\,A^{\bar{0}}$ take this form.
\end{proposition}

Given an additional assumption, which holds in the examples we consider, irreducible objects in $\mathrm{Rep}\,\mcA_{L^{\mu}}$ are given by the action of the induction functor $\mcF:\mcB \to \mathrm{Rep}\mcA_{L^{\mu}}$ (Definition \ref{indfunctor}) on irreducibles in $\mcB$ (Proposition \ref{simpleinduction}):
\begin{proposition} Suppose every indecomposable object in $\mcB$ has a simple subobject. Then
$N \in \mathrm{Rep}^0\mcA_{L^{\mu}}$ is simple if and only if $N \cong \mcF(M)$ for a simple object $M \in \mcB$.
\end{proposition}

The primary example we consider is the category of weight modules $\mcC$ over the unrolled restricted quantum groups (see Subsection \ref{subunrolled}) at root of unity $q$ of order $\ell \geq 3$. The simple objects in $\mcC$ are denoted by $\irred{\lambda}$ with $\lambda \in \mfh^*$ where $\mfh=\mathrm{Span}\{H_1,...,H_n\} \subset \overline{U}_q^H(\mfg)$. This category admits projective covers for each irreducible module, which we denote by $P^{\lambda}$, and we have the following results for induction of projective and irreducible modules (Theorem \ref{induction}):
\begin{theorem}
${}$
\begin{itemize}
\item $\mcF(P^{\lambda}) \in \mathrm{Rep}^0(\mcA_{L^{\mu}})$ if and only if $\lambda \in \frac{\ell}{2}(L^{\mu})^*$.
\item Let $X \in \C$ and let $P_X \in C$ be the projective cover of $X$. Then $\mcF(X) \in \mathrm{Rep}^0(\mcA_{L^{\mu}})$ if and only if $\mcF(P_X) \in \mathrm{Rep}^0\mcA_{L^{\mu}}$.
\item $\mcF(P^{\lambda})$ is the projective cover of $\mcF(\irred{\lambda})$ in $\mathrm{Rep}^0\mcA_{L^{\mu}}$.
\item The distinct irreducible objects in $\mathrm{Rep}^0\mcA_{L^{\mu}}$ are $\{ \mcF(\irred{\lambda}) \, | \, \lambda \in \Lambda(L^{\mu})\}$, where $\Lambda(L^{\mu}):=\frac{\ell}{2}(L^{\mu})^*/\frac{\ell}{2}(L^{\mu})^* \cap L^{\mu}$. $\mathrm{Rep}^0\mcA_{L^{\mu}}$ is finite if and only if $\mathrm{rank}(L^{\mu})=\mathrm{rank}(P)$.
\end{itemize}
\end{theorem}
Further, we can determine when the category of local modules is ribbon and has trivial M\"{u}ger center (see Definition \ref{muger}, Proposition \ref{rib}):

\begin{proposition}\label{rib2}
Let $\mcA_{L^{\mu}}$ be a supercommutative superalgebra, then 
\begin{itemize}
\item $\mathrm{Rep}^0\mcA_{L^{\mu}}$ is ribbon if $2(1-r)\langle \lambda, \rho \rangle \in \ell \mathbb{Z}$ for all $\lambda \in L$ and $2(1-r)\langle \mu, \rho \rangle \in \frac{\ell}{2}\mbbZ$.
\item Let $r=2\ell/(3+(-1)^{\ell})$. If $r\nmid 2d_i$, for all $i$, then $\mathrm{Rep}^0\mcA_{L^{\mu}}$ has non-trivial M\"{u}ger center if and only if there exists a $\lambda\in \Lambda(L^{\mu})$ such that $\langle \lambda, \gamma \rangle \in \frac{\ell}{2} \mbbZ$ for all $\gamma \in \Lambda(L^{\mu})$.
\end{itemize}
\end{proposition}
Note that the results for commutative algebra objects are given by setting $\mu=0$. In subsections \ref{subtriplet} and \ref{subbp}, we construct algebra objects $\mcA_{rQ}$ and $\mfB^{a_r}_{rP}$ whose categories of local modules are expected to compare nicely to particular module categories of $W_Q(r)$ and $B_Q(r)$ respectively. The structure of their categories of local modules are as follows (see Propositions \ref{tripletao} and \ref{bpao}):

\begin{proposition}${}$
\begin{itemize}
\item $\mathrm{Rep}^0\mcA_{rQ}$ is a finite non-degenerate ribbon category (i.e. Log-Modular) with $\mathrm{det}(A)\cdot r^{\mathrm{rank}(\mfg)}$ distinct irreducible modules, where $A$ is the Cartan matrix of $\mfg$.
\item $\mathrm{Rep}^0\mfB_{rP}^{a_r}$ is non-degenerate, ribbon if r is odd or $\rho \in Q$, and the irreducible modules are
\[ \{ \mcF(\irred{\mu} \boxtimes \mathsf{F}_{\gamma}) \; | \,  \mu,\gamma \in \mfh^*, \, \mathrm{and} \, \mu+ra_r\gamma \in Q \; \} \] 
with relations $\mcF( \irred{\mu}\boxtimes \mathsf{F}_{\gamma}) \cong \mcF(\irred{\mu+\lambda} \boxtimes \mathsf{F}_{\gamma+a_r\lambda})$ for all $\lambda \in rP$, where $a_r=\sqrt{-1/r}$.
\end{itemize}
\end{proposition}

\subsection{Connections to vertex operator algebras}

It is generally expected that sufficiently nice categories of vertex algebra modules form a rigid braided tensor category. This, however, is only known to be true for the simplest class of vertex algebras, those which are strongly rational. In this case one even has the structure of a modular tensor category, e.g. a finite, semi-simple, non-degenerate ribbon tensor category \cite{H1, H2}. For module categories of vertex operator algebras which are not semi-simple, already the existence of a braided tensor category structure is difficult to prove, let alone rigidity. Existence for $C_2$-cofinite vertex algebras is due to Huang as well \cite{H3}, these are conjecturally log-modular, i.e. finite and non-degenerate \cite{CGan}.
The main further known examples are rigid vertex tensor category structures on the category of ordinary modules of an affine Lie algebra at generic level \cite{KL1, KL2, KL3, KL4}, admissible level and few other non-generic levels (rigidity only for $\mathfrak g$ simply-laced) \cite{CHY, C3, CY} and the Virasoro algebra \cite{CJHRY}. Given the exciting properties of strongly rational vertex algebras, such as modularity of representation categories and Verlinde's formula \cite{H1}, one hopes to also find a lot of interesting structure in the non semi-simple cases \cite{CGan}. However, we are lacking the accessible examples required to explore general properties of non-rational vertex operator algebras.

The prototypical examples of strongly rational vertex algebras are affine vertex algebras at positive integer level \cite{FZ} and principal $\mathcal W$-algebras at non-degenerate admissible levels \cite{Ara}. The non-rational vertex algebras one wants to understand are affine vertex (super)algebras and general $\mathcal W$-(super)algebras at admissible level and beyond. However, only the case of $\mathfrak{sl}_2$ is fairly well studied and even for vertex operator algebras associated to Lie algebras of rank $2$, we are lacking instructive examples. Examples of non-rational vertex operator algebras associated to $\mathfrak{sl}_2$ are the singlet, triplet and $B_p$-algebras of \cite{Kau, A, AM, TW, CRW} and their representation categories are conjecturally closely related to those of restricted and unrolled quantum $\mathfrak{sl}_2$ at even root of unity. 

The higher rank analogues of singlet, triplet and $B_p$-algebra have been introduced in \cite{CM, FT, C1} and are denoted by $W^0_Q(r)$, $W_Q(r)$, and $B_Q(r)$ respectively where $Q$ is the root lattice of a simple finite dimensional complex Lie algebra $\mfg$ of ADE type and $r \in \mbbZ_{\geq 2}$. Studies are notorously difficult so most satements are given conjecturally and only very recently Shoma Sugimoto proved some of the Feigin-Tipunin conjectures \cite{Su}; and some more results for $W_{A_2}^0(2)$ just appeared \cite{AMW}. The category of modules $\mathrm{Rep}_{\langle s \rangle}W^0_Q(r)$ generated by irreducible $W^0_Q(r)$ modules is expected to be ribbon equivalent to the category of weight modules $\mathrm{Rep}\,\overline{U}_q^H(\mfg)^{wt}$ for the unrolled restricted quantum group $\overline{U}_q^H(\mfg)$ at $2r$-th root of unity $q$, motivated largely by already established connections in the $\mfg=\mathfrak{sl}_2$ case. The fusion rules for $W^0_{A_1}(r)$ are known only for $r=2$, but there is a conjecture for $r>2$ \cite{CM2}. It has been shown (see \cite{CMR}) that if the fusion rules are as conjectured, then there exists an identification $\phi:\mathrm{Irr}(\mathrm{Rep}\overline{U}_q^H(\mathfrak{sl}_2)^{wt}) \to \mathrm{Irr}(\mathrm{Rep}_{\langle s \rangle} W^0_{A_1}(r))$ between sets of irreducible modules in each category which induces an isomorphism of Grothendieck rings, and can be extended to indecomposables in a way which preserves Loewy diagrams \cite{CGR}. We also know \cite[Theorem 1]{CMR} that the regularized asymptotic dimensions of irreducible $W^0_{A_1}(p)$-modules coincide exactly with normalized modified traces of open Hopf links for the corresponding (under $\phi$) $\overline{U}_q^H(\mathfrak{sl}_2)$-module. The identification $\phi$ can be used to construct braided tensor categories which compare nicely to certain categories of modules of the triplet and $B_p$ vertex operator algebras, as in \cite{CGR} and \cite{ACKR} respectively.

Extending the connections between $\overline{U}_q^H(\mathfrak{sl}_2)$ and the rank one singlet, triplet, and $B_p$ vertex operator algebras to higher rank is one of our primary motivations. The $B_{A_2}(2)$-algebra is particularly interesting, as it is the simple affine vertex algebra of $\mathfrak{sl}_3$ at level $-3/2$, $L_{-\frac{3}{2}}(\mathfrak{sl}_3)$ \cite{A2}. In forthcoming work, we will study the representation theory of the  corresponding algebra object of the unrolled quantum group of $\mathfrak{sl}_3$ in detail and compare them to known structure of $L_{-\frac{3}{2}}(\mathfrak{sl}_3)$. As described in Subsections \ref{subtriplet} and \ref{subbp}, the vertex operator algebras $W_Q(r)$ and $B_Q(r)$ can be identified with commutative algebra objects in $(\mathrm{Rep}\overline{U}_q^H(\mathfrak{sl}_2)^{wt})^{\oplus}$ and $(\mathrm{Rep}\overline{U}_q^H(\mathfrak{sl}_2)^{wt} \boxtimes \mcH)^{\oplus}$ respectively where $\mcH$ is the category of modules over the Heisenberg vertex operator algebra on which the horizontal subalgebra acts semisimply (see Subsection \ref{Heis}), and $\mcC^{\oplus}$ denotes the direct sum completion of $\mcC$ (see \cite{AR}). We expect that the module categories $\mathrm{Rep}_{\langle s \rangle} W_Q(r)$ and $\mathrm{Rep}_{\langle s \rangle}B_Q(r)$ generated by irreducible $W_Q(r)$ and $B_Q(r)$ modules respectively to be ribbon equivalent to the categories of local modules over their corresponding algebra objects, which is motivated by the $\mfg=\mathfrak{sl}_2$ case \cite{ACKR,CGR}. The algebra objects corresponding to $W_Q(r)$ and $B_Q(r)$ are denoted $\mcA_{rQ}\in \mcC^{\oplus}$ and $\mfB_{rP}^{a_r}\in (\mcC \boxtimes \mcH)^{\oplus}$ respectively, and the structure of their categories of local modules are given in Propoisitions \ref{tripletao} and \ref{bpao}. Using this identification, we make the following conjectures for the corresponding categories of vertex operator algebra modules:

\begin{conjecture}${}$
\begin{enumerate} 
\item $\mathrm{Rep}_{\langle s \rangle}\mcW_Q(r)$ is a finite non-degenerate ribbon category with $\mathrm{det}(A)\cdot r^{\mathrm{rank}(\mfg)}$ distinct irreducible modules, where $A$ is the Cartan matrix of $\mfg$.
\item $\mathrm{Rep}_{\langle s \rangle} B_Q(r)$ is non-degenerate, ribbon if $r$ is odd or $\rho \in Q$, and the irreducible modules can be indexed as
\[ \{ \irred{\mu}_{\gamma} \; | \,  \mu,\gamma \in \mfh^*, \, \mathrm{and} \, \langle \lambda, \mu+ra_r\gamma \rangle \in \mbbZ \; \text{for all $\lambda \in P$}\}\] 
with relations $\irred{\mu}_{\gamma} \cong \irred{\mu+\lambda}_{\gamma+a_r\lambda}$ for all $\lambda \in rP$, where $a_r=\sqrt{-1/r}$.
\end{enumerate}
\end{conjecture}
We remark, that the number of irreducibles in the conjecture for $W_Q(r)$ agrees with observations of Shoma Sugimoto that he presented in a seminar talk.

\subsection{Connections to physics}

Vertex algebras are intimately connected to physics as the chiral algebra of a two-dimensional conformal quantum field theory is a vertex operator algebra. If correlation functions in the field theory exhibit logarithmic singularities then the theory is called a logarithmic conformal field theory, see \cite{CR} for an introduction. This behaviour is tied to the nilpotent action of the Virasoro zero-mode and thus leads to vertex algebras with non semi-simple representation categories. The very first well-studied example were the symplectic fermions whose even subalgebra is the triplet algebra at $r=2$ \cite{Kau2}. 

The modern relevance of vertex algebras and their representation categories in physics is as meaningful invariants of three and four-dimensional supersymmetric field theories \cite{Beem}. This is also mathematically far reaching as it connects to invariants of 3- and 4-manifolds \cite{FeGu} and the quantum geometric Langlands program \cite{CGa, FrGa}.  In the latter case vertex algebras are supposed to be over a localization of $\mathbb C[\psi]$, where $\psi$ is the coupling of the gauge theory. The $W_Q(r)$-algebras appear in certain configurations as large coupling limits and so does the small $N=4$ superconformal algebra at central charge $-9$ \cite{CGL}. But the latter can be realized as a simple current extension of a rank one Heisenberg vertex algebra times $W^0_{A_2}(2)$ \cite{A2, CGL} and its quantum group analogue is thus also covered by our techniques. 
The $B_Q(r)$-algebras and their generalizations are conjecturally chiral algebras of Argyres-Douglas theories \cite{C1}. Physics knows central charge \cite{XZ}, affine subalgebra and character \cite{BN} of the chiral algebra and the conjecture is based on the $B_Q(r)$-algebras exactly having these properties. 
It has been recently proven for rank one that these properties uniquely determine the vertex algebra and hence in this case the conjecture is true \cite{ACGY}. Argyres-Douglas theories \cite{AD} are four-dimensional supersymmetric gauge theories associated to pairs of Dynkin diagrams (usually of ADE-type) and they serve as a rich source of examples. Much of the physics data has nice representation categorical interpretation, e.g. line and surface defects correspond to modules \cite{CGS1, CGS2},  modular data gives Wild Hitchin characters \cite{FPYY} and in general modular data and fusion rules are visible as algebras of line defects \cite{NY}. Note that in general these categories of line defects even go beyond the category of weight modules that we consider in this work 
\cite{CCG}. In \cite{CCG} also Ext-algebras of the unrolled and uprolled quantum groups for $\mathfrak{sl}_2$ were computed and results suggest an SU(2)-action in the uprolled case, i.e. on Ext$_{ \mathrm{Rep}^0(\mcA_{rA_1})}^\bullet(\mcA_{rA_1}, \mcA_{rA_1})$.  This hints to an exciting connection to geometry. Firstly, note that the Feigin-Tipunin definition of $W_Q(r)$ is as global sections of a sheaf of vertex algebras on the flag manifold of the Lie group $G$ whose Lie algebra $\mathfrak g$ has $Q$ as root lattice. Secondly and at least if  $G$ is simply-laced then $G$ acts on  $W_Q(r)$ by automorphisms \cite{Su}. On the other hand for odd roots of unity a correspondence between Ext-algebras of quantum groups and 
of coherent sheaves on the cotangent bundle of the flag manifold of $G$ is known \cite{ABG}. Moreover the Ext-algebras carry an action of $G$. 
A future aim is thus to study Ext$_{ \mathrm{Rep}^0(\mcA_{rQ})}^\bullet(\mcA_{rQ}, \mcA_{rQ})$ and ideally to connect the results to some geometry. 

\noindent {\bf Acknowledgements} We thank Shoma Sugimoto and Boris Feigin for discussions related to this work. 
T. C. is supported by NSERC RES0048511.

\section{Preliminaries}

The primary objects of study in this article are commutative algebra objects and supercommutative superalgebra objects constructed as direct sums of simple currents in the category of weight modules for unrolled restricted quantum groups. We review in this section the relevant material for their construction. A good reference is \cite{KO} and for superalgebra results see \cite{CKM}.

\subsection{Algebra Objects and Simple Currents}\label{subcurrents}

\begin{definition}
A simple current is a simple object which is invertible with respect to the tensor product. Objects which are their own inverse are called self-dual.
\end{definition}

\begin{definition}\label{alg} A commutative associative unital algebra in a braided tensor category $\C$ is an object $A \in \C$ with $\C$-morphisms $\mu \in \mathrm{Hom}(A \otimes A, A)$, $\iota \in \mathrm{Hom}(\mathds{1},A)$ satisfying the following conditions:
\begin{itemize}
\item Associativity: $\mu \circ (\mu \otimes \Id_A)=\mu \circ ( \Id_A \otimes \mu) \circ a_{A,A,A}$ where $a_{A,A,A}: (A \otimes A) \otimes A \to A \otimes (A \otimes A)$ is the associativity isomorphism.
\item Unit: $\mu \circ (\iota \otimes \Id_A) \circ l_A^{-1}=\Id_A$ where $l_A:\mathds{1} \otimes A \rightarrow A$ is the left unit isomorphism.
\item Commutativity: $\mu\circ c_{A,A} = \mu$ where $c_{A,A}$ is the braiding.
\item (Optional assumption) Haploid: $\dim(\Hom_{\mathcal{C}}(\mathds{1},A))=1$.\\
\end{itemize}
\end{definition}
Supercommutative superalgebra objects are defined as follows \cite{CKL}:
\begin{definition}\label{salg}
$A \in \C$ is a superalgebra if it is an algebra with a $\mbbZ_2$-grading compatible with the product. That is, $A=A^{\bar{0}} \oplus A^{ \bar{1}}$ such that 
\[ \mu(A^{\bar{i}} \otimes A^{\bar{j}}) \subset A^{\overline{i+j}}.\]
$A$ is supercommutative if for $x\in A^{\bar{i}},y \in A^{\bar{j}}$, 
\begin{equation}\label{supercom} \mu|_{A^{\bar{i}} \otimes A^{\bar{j}}}= (-1)^{ij} \cdot \mu \circ c_{A^{\bar{i}},A^{\bar{j}}}.
\end{equation}
\end{definition}

We denote by $\mathrm{Rep}\,A$ the category of objects $(V, \mu_V)$ where $V \in \C$ is an object in $\C$ and $\mu_V \in \mathrm{Hom}(A \otimes V, V)$ is a $\C$-morphism satisfying the usual assumptions required to make $V$ an $A$-module:

\begin{itemize}
\item $\mu_V \circ (\Id_A \otimes \mu_V)=\mu_V \circ ( \mu \otimes \Id_A)\circ a_{A,A,V}^{-1}$,
\item $\mu_V \circ (\iota \otimes \Id_V) \circ l_V^{-1}=\Id_V$.
\end{itemize}

Although $\C$ is braided, $\mathrm{Rep}\,A$ need not be, but there is a full subcategory of $\mathrm{Rep}\,A$ which is braided \cite{KO}:

\begin{definition}\label{local}
The category of local modules $\mathrm{Rep}^0A$ is the full subcategory of $\mathrm{Rep}\,A$ whose objects are given by 
\[ \left\{ (V,\mu_V) \in \mathrm{Rep}\,A \, | \, \mu_V \circ c_{V,A} \circ c_{A,V} = \mu_V \right\} \]
\end{definition}

The most well-behaved modules in $\mathrm{Rep}\,A$ are those which can be obtained via the induction functor:
\begin{definition}\label{indfunctor}
The induction functor $\mathcal{F}:\C \to \mathrm{Rep}\,A$ is defined by $\mathcal{F}(V)=(A \otimes V, \mu_{\mathcal{F}(V)})$, with $\mu_{\mathcal{F}(V)}=( \mu \otimes \mathrm{Id}_V) \circ a_{A,A,V}^{-1}$ and $\mathcal{F}(f)=Id_A \otimes f$, where $\mu:A \otimes A \to A$ is the product on $A$, and $a_{-,-,-}$ the associativity morphism in $\C$.
\end{definition}

We also have the forgetful restriction functor $\mathcal{G}:\mathrm{Rep}\,A \to \C$ given by $(V,\mu_V) \mapsto V$. The induction and restriction functors satisfy Frobenius reciprocity:
\begin{equation}\label{reciprocity}
\mathrm{Hom}_{\C}(X,\mathcal{G}(Y)) \cong \mathrm{Hom}_{\mathrm{Rep}\,A}(\mcF(X),Y) 
\end{equation}
for all $X \in \C$, $Y \in \mathrm{Rep}\,A$. \\

If $\mcC$ has braiding $c_{-,-}$, then an object $Y \in \mcC$ is said to be transparent if $c_{Y,X} \circ c_{X,Y} = \mathrm{Id}_{X \otimes Y}$ for all $X \in \mcC$. 
\begin{definition}\label{muger}
The M\"{u}ger center of $\C$ is the full subcategory of $\C$ consisting of all transparent objects.
\end{definition}
For finite braided tensor categories, triviality of the M\"{u}ger center is equivalent to the usual notions of non-degeneracy (see \cite[Theorem 1.1]{S}).

\subsection{Unrolled Quantum Groups}\label{subunrolled}

Let $\mfg$ be a simple finite-dimensional complex Lie algebra of rank $n$ and dimension $n+2N$ with Cartan matrix $A=(a_{ij})_{i,j=1}^n$ and Cartan subalgebra $\mfh$. Let $\Delta:=\{\alpha_1,...,\alpha_n\} \subset \mfh^*$ be the set of simple roots of $\mfg$, $\Delta^+$ ($\Delta^-$) the set of positive (negative) roots, and $Q:=\bigoplus\limits_{i=1}^n \mbbZ \alpha_i$ the integer root lattice. Let $\{H_1,...,H_n\}$ be the basis of $\mfh$ such that $\alpha_j(H_i)=a_{ij}$ and $\langle , \rangle$ the form defined by $\langle \alpha_i,\alpha_j \rangle=d_ia_{ij}$ where $d_i=\langle \alpha_i,\alpha_i\rangle/2$ and normalized such that short roots have length 2. Let $P:= \bigoplus\limits_{i=1}^n \mbbZ \omega_i$ be the weight lattice generated by the dual basis $\{\omega_1,...,\omega_n\} \subset \mfh^*$ of $\{d_1H_1,...,d_nH_n\} \subset \mfh$, and $\rho:= \frac{1}{2} \sum\limits_{\alpha \in \Delta^+} \alpha \in P$ the Weyl vector.\\

Now, let $\ell \geq 3$ such that $r=2\ell/(3+(-1)^{\ell}) > \mathrm{max}\{g_1,...,g_n\}$ where $g_k=\mathrm{gcd}(d_k,r)$. Let $q \in \mathbb{C}$ be a primitive $\ell$-th root of unity, $q_i=q^{d_i}$, and fix the notation
\begin{equation}\{x\}=q^x-q^{-x}, \quad  [x]=\frac{q^x-q^{-x}}{q-q^{-1}}, \quad [n]!=[n][n-1]...[1], \quad \binom{n}{m}=\frac{\{n\}!}{\{m\}!\{n-m\}!} ,
\end{equation}
\begin{equation}
 [j;q]=\frac{1-q^j}{1-q}, \qquad [j;q]!=[j;q][j-1;q] \cdots [1;q] \label{jq}.
\end{equation}
We will often use a subscript $i$, e.g. $[x]_i$, to denote the substitution $q \mapsto q_i$ in the above formulas. Fix a lattice $\mathsf{L}$ such that $Q \subset \mathsf{L} \subset P$ and let $r_i:=r/\mathrm{gcd}(d_i,r)$.

\begin{definition}\label{Uqh}
The unrolled restricted quantum group associated to $\mfg$ at root of unity $q$ is the $\mbbC$-algebra with generators $X_{\pm i}, H_i, K_{\gamma}$ with $i=1,...,n$, $\gamma \in \mathsf{L}$, and relations
\begin{align}
\label{eqKX}K_0=1, \qquad K_{\gamma_1}K_{\gamma_2}=K_{\gamma_1+\gamma_2},&\qquad K_{\gamma}X_{\pm j}K_{-\gamma}=q^{\pm\langle \gamma,\alpha_j \rangle}X_{\sigma j},\\
[H_i,X_{\pm j}]=\pm a_{ij} X_{\pm j }, &\quad [H_i,K_{\gamma}]=0, \quad [H_i,H_j]=0\\
\label{eqX} X_{\pm i}^{r_i}=0  \qquad [X_i,X_{-j}]&=\delta_{i,j}\frac{K_{\alpha_j}-K_{\alpha_j}^{-1}}{q_j-q_j^{-1}},\\
\label{serre} \sum\limits_{k=0}^{1-a_{ij}} (-1)^k \binom{1-a_{ij}}{k}_{q_i} X^k_{\pm i}&X_{\pm j} X_{\pm i}^{1-a_{ij}-k}=0 \qquad \text{if $ i \not = j$}.
\end{align}
There is a Hopf-algebra structure on $\overline{U}_{q}^H(\mfg)$ with coproduct $\Delta$, counit $\epsilon$, and antipode $S$ defined by
\begin{align}
\label{coK}\Delta(K_{\gamma})&=K_{\gamma} \otimes K_{\gamma}, & \epsilon(K_{\gamma})&=1,&  S(K_{\gamma})&=K_{-\gamma}\\
\label{coXi}\Delta(X_i)&=1 \otimes X_i + X_i \otimes K_{\alpha_i} & \epsilon(X_i)&=0, & S(X_i)&=-X_iK_{-\alpha_i},\\
\label{coX-i}\Delta(X_{-i})&=K_{-\alpha_i} \otimes X_{-i}+X_{-i} \otimes 1, & \epsilon(X_{-i})&=0, & S(X_{-i})&=-K_{\alpha_i}X_{-i}.\\
\label{coH} \Delta(H_i)&=1\otimes H_i + H_i \otimes 1,& \quad \epsilon(H_i)&=0, &\quad S(H_i)&=-H_i.
\end{align}

\begin{definition}\label{weightmod}
We call a $\overline{U}_q^H(\mfg)$-module $V$ a weight module if $K_{\gamma}=\prod\limits_{i=1}^n q^{d_ic_iH_i}$ for $\gamma=\sum\limits_{i=1}^n c_i\alpha_i \in \mathsf{L}$ as operators on $V$ and $V$ decomposes as a direct sum of eigenspaces
\[V=\bigoplus\limits_{\lambda \in \mfh^*} V(\lambda)\]
where $V(\lambda)=\{ v \in V \, | \, H_iv=\lambda(H_i)v \}$. We denote by $\mcC$ the category of finite dimensional weight modules.
\end{definition}

Define the operators $\mathscr{H}$ and $\mathscr{R}$ on $\C$ by the action of 
\begin{equation}\label{braidingH} \mathscr{H}:=q^{\sum\limits_{i,j=1}^n d_i(A^{-1})_{ij}H_i \otimes H_j}\end{equation}
and 
\begin{equation}\label{braidingR} \mathscr{R}:= \prod\limits_{i=1}^N\left( \sum \limits_{j=1}^{r_{\beta_i}-1} \frac{\left( (q_{\beta_i} -q_{\beta_i}^{-1})X_{\beta_i} \otimes X_{-\beta_i}\right)^j}{[j;q_{\beta_i}^{-2}]!} \right) 
\end{equation}
where we recall that $A_{ij}$ is the Cartan matrix of $\mfg$ and $r_{\beta_i}=r/\mathrm{gcd}(d_{\beta_i},r)$. One can easily modify the argument in \cite[Subsection 5.8]{GP1}, as described in \cite[Subsection 4.1]{R}, to show that $\C$ is braided with braiding given by the action of the operator $\tau \circ \mathcal{R}$ where $\mathcal{R}=\mathscr{H}\mathscr{R}$ and $\tau(x \otimes y)=y \otimes x$ is the usual flip map. Given vectors $v_{\lambda},v_{\mu} \in V$ such that $H_iv_{\lambda}=\lambda(H_i)v_{\lambda}, H_iv_{\mu}=\mu(H_i)v_{\mu}$ (i.e. weight vectors of weight $\lambda,\mu \in \mfh^*$), $\mathscr{H}$ acts as
\begin{equation} \label{weight}
\mathscr{H}(v_{\lambda} \otimes v_{\mu})=q^{\langle \lambda, \mu \rangle} v_{\lambda} \otimes v_{\mu}.
\end{equation}
\end{definition}
The twist on $\mcC$ is given by the right partial trace of the braiding, and acts on an irreducible module $\irred{\lambda}$ of highest weight $\lambda \in \mfh^*$ as \cite[Equation 4.5]{R}:
\begin{equation}\label{twist}
\theta_{\irred{\lambda}}=q^{\langle \lambda,\lambda+2(1-p)\rho \rangle}Id_{\irred{\lambda}}.
\end{equation}

Connections between the the representation theory of $\overline{U}_q^H(\mathfrak{sl}_2)$ and the singlet vertex operator algebra $W^0_{A_1}(r)$ were established in \cite{CMR}. The identification \cite[Theorem 1]{CMR} between irreducible $\overline{U}_q^H(\mathfrak{sl}_2)$ and $W^0_{A_1}(r)$ modules preserved twists. 
The higher rank singlet $W^0_Q(r)$ (see \cite{CM} for details) is a subalgebra of the Heisenberg vertex algebra of rank is the rank of $Q$ and for $\lambda \in \sqrt{r}P$ it is expected that $F_\lambda$ has a simple $W^0_Q(r)$-submodule $M_\lambda$ that is a simple current. This Conjecture would follow from \cite[Cor. 4.8]{McRae} together with \cite{Su} if one could prove that $W_Q(r)$ is a simple vertex algebra and braided tensor category exists on a category of $W^0_Q(r)$ that contains all the $M_\lambda$. 

The twist is determined from the action of $e^{2\pi i L_0}$ with $L_0$ the Virasoro zero-mode of the vertex algebra. 
The twist acts on  Fock spaces $F_{\lambda}$ by the scalar $\theta_{F_{\lambda}}=e^{\pi i \langle \lambda, \lambda + \frac{2(1-r)}{\sqrt{r}}\rho \rangle}$. It then follows from Equation \ref{twist}, that the twists on $\mathrm{Rep}_{\langle s \rangle} W^0_Q(r)$ and $\mathrm{Rep}\,\overline{U}_q^H(\mfg)^{wt}$ act on $F_{\lambda}$ and $\irred{\sqrt{r}\lambda}$ by the same scalar. We therefore make the identification
\begin{equation}\label{id}
\phi: \mathrm{Rep}_{\langle s \rangle} W^0_Q(r) \mapsto \mathrm{Rep}\overline{U}_q^H(\mfg)^{wt}, \qquad M_{\lambda} \subset F_\lambda \to \irred{\sqrt{r} \lambda} 
\end{equation}
for $\lambda \in \sqrt{r}P$.
Note that the simple currents in $\mathrm{Rep}_{\langle s \rangle} W^0_Q(r)$ are conjecturally precisely the Fock spaces $M_{\lambda}\subset F_\lambda$ with $\lambda \in \sqrt{r}P$ while the simple currents in $\mathrm{Rep}\,\overline{U}_q^H(\mfg)^{wt}$ (for $\mfg$ of ADE type) are the irreducibles $\irred{\lambda}$ with $\lambda \in rP$ (see \cite[Remark 4.7]{R}), so $\phi$ identifies simple currents.

\subsection{Heisenberg vertex operator algebra}\label{Heis}

One of the examples we are interested in is the Deligne product $\mcC \boxtimes \mcH$ of the category $\mcC$ of $\overline{U}_q^H(\mfg)$ weight modules with a semi-simple category $\mcH$ of modules over the Heisenberg vertex operator algebra. We define and review the category $\mcH$ in this section following \cite{FB}. Tensor category structure is due to \cite{CKLR}.

Let $\mfh$ be the Cartan subalgebra of $\mfg$ and $\hat{\mfh}= \mfh\otimes \mbbC[t,t^{-1}] \oplus \mbbC K$ the corresponding affine Lie algebra. Let $\lambda \in \mfh^*$ and denote by $\mathsf{F}_{\lambda}$ the the usual Fock space
\[ \mathsf{F}_{\lambda}:= U(\hat{\mfh}) \otimes_{U(\mfh \otimes \mbbC[t] \oplus \mbbC K)}\mbbC \]
where $\mfh \otimes t\mbbC[t]$ acts trivially on $\mbbC$, $\mfh$ acts as $\lambda(h)$ for all $h \in \mfh$, and $K$ acts as 1. For $h \in \mfh$ and $n \in \mbbZ$ we adopt the notation $h(n):=h \otimes t^n \in \mfh \otimes \mbbC[t,t^{-1}]$ and define
\[ h(z):= \sum\limits_{n \in \mbbZ} h(n)z^{-n-1}.\]
The Fock space $\mathsf{H}:=\mathsf{F}_0$ carries the structure of a vertex operator algebra. Set $\mathsf{1}:= 1\otimes 1$ and for any $v:=h_1(-n_1) \cdots h_m(-n_m) \mathsf{1} \in \mathsf{H}$ where $h_1,...,h_m \in \mfh$, we define the vertex operator $Y(-,z)$ acting on $\mathsf{F}_{\lambda}$  by
\[ Y(v,z):= \colon [ \partial^{n_1-1}h_1(z)] \cdots [\partial^{n_m-1}h_m(z)] \colon,  \]
where $\partial^k=\frac{1}{n!}(\frac{d}{dz})^k$ and $\colon XY \colon$ denotes the normal ordering of two fields $X,Y$. Set $\omega:= \frac{1}{2}\sum\limits_{k=1}^n H_k(-1)^2\mathsf{1}$ where $H_1,...,H_n$ is an orthonormal basis of $\mfh$ with respect to the Killing form, and we denote $Y(\omega,z)=\sum\limits_{n\in \mbbZ} L_nz^{-n-1}$. Then, $\mathsf{H}:=\mathsf{F}_0$ is a vertex operator algebra and it has vertex tensor categories of Fock modules, but for that one has to ensure that conformal weight is real (see \cite[Theorem 2.3]{CKLR}, which requires the action of the $H_k$ to be either real or purely imaginary. For us the latter case is relevant, e.g.

$(\mathsf{H},Y,\mathds{1},\omega)$ is a simple vertex operator algebra called the Heisenberg vertex operator algebra with vacuum $\mathsf{1}$, state-field correspondence $Y(-,z)$, and Virasoro element $\omega$. 
\begin{definition}\label{Hcat}
We define $\mcH$ to be the category of $\mathsf{H}$-modules on which $\mathfrak h$ acts semisimply and $\lambda(H_k) \in i \mathbb R$ for all $k=1, \dots, n$. 
\end{definition}
This category is semisimple and generated by the Fock modules $\mathsf{F}_{\lambda}$, $\lambda \in \mfh^*$. Tensor products are additive in the index ($\mathsf{F}_{\lambda_1} \otimes \mathsf{F}_{\lambda_2} \cong \mathsf{F}_{\lambda_1+\lambda_2}$), so all Fock spaces are simple currents with $\mathsf{F}_{-\lambda}$ the inverse of $\mathsf{F}_{\lambda}$. The braiding and twist on $\mcH$ are given by
\begin{align}
\label{Hbraid} \mathrm{Braiding}& & c_{\mathsf{F}_{\lambda_1},\mathsf{F}_{\lambda_2}}&=\tau \circ e^{\pi i \langle \lambda_1, \lambda_2 \rangle}\\
\label{Htwist} \mathrm{Twist} & & \theta_{\mathsf{F}_{\lambda}}&=e^{\pi i \langle \lambda, \lambda \rangle}\mathrm{Id}_{\mathsf{F}_{\lambda}}
\end{align}
where $\tau$ is the usual flip map.

\section{Simple Current Extensions}

\setlength\parindent{0pt}
\setlength{\parskip}{\baselineskip}%

We want to construct and study braided tensor categories related to the module categories of the higher rank $W_Q(r)$ and $B_Q(r)$ vertex operator algebras. As described in the introduction, such categories can be constructed as categories of local modules for algebra objects in the category of weight modules $\mathrm{Rep}\overline{U}_q^H(\mfg)^{wt}$ and the Deligne product $\mathrm{Rep}\overline{U}_q^H(\mfg)^{wt} \boxtimes \mcH$ where $\mcH$ is defined in Definition \ref{Hcat}. Both cases can be handled simultaneously by considering categories of a particular form. Throughout this section, let $\mathcal{B}$ be a braided tensor category whose simple currents $\{\mbbC_{\lambda} \, | \, \lambda \in \mathcal{L}\}$ are indexed by a normed lattice $(\mathcal{L},\langle -,- \rangle)$ such that the braiding on simple currents is given by $c_{\mbbC_{\lambda_1},\mbbC_{\lambda_2}}(v_{\lambda_1} \otimes v_{\lambda_2})=q^{\langle \lambda_1, \lambda_2 \rangle}v_{\lambda_2} \otimes v_{\lambda_1}$ for a primitive $\ell$-th root of unity q and $\mbbC_{\lambda_1} \otimes \mbbC_{\lambda_2} \cong \mbbC_{\lambda_1+\lambda_2}$ where $\mbbC_0$ is the unit object in $\mcB$. We assume also that objects in $\mcB$ carry vector space structure over $\mbbC$ and $\mathrm{dim}\mathrm{End}(\mbbC_{\lambda})=1$ for all $\lambda \in \mcL$. For any lattice $L \subset \mathcal{L}$, define the object 
\[ \mathcal{A}_L:=\bigoplus\limits_{\lambda \in L} \mbbC_{\lambda} \in \mcB^{\oplus}.\]
where $\mcB^{\oplus}$ denotes an appropriate direct sum completion of $\mcB$ (see \cite{AR}).

\begin{theorem}\label{algebraobject}
$\mcA_L$ is an associative algebra object for all $L \subset \mathcal{L}$. $\mcA_L$ is commutative if and only if  $\sqrt{2/\ell}L$ is an even lattice. That is, if and only if $\langle \lambda,\lambda \rangle \in \ell \mbbZ$ and $2\langle \lambda,\mu \rangle \in \ell \mbbZ$ for all $\lambda,\mu \in L$.
\end{theorem}

\begin{proof}

To realize $\mcA_{L}$ as a commutative algebra object in $\mcB^{\oplus}$, we must define a product $\mu:\mcA_{L} \otimes \mcA_{L} \to \mcA_{L}$ and unit $\iota:\mbbC_0 \to \mcA_{L}$ satisfying the associativity, unit, and commutativity constraints:
\begin{align}
\label{associative} \mu \circ (\mu \otimes \Id_{\mcA_{L}})&=\mu \circ ( \Id_{\mcA_{L}}\otimes \mu) \circ a_{\mcA_{L},\mcA_{L},\mcA_{L}}, \\
\label{unit}\mu \circ (\iota \otimes \Id_{\mcA_{L}}) \circ l_{\mcA_L}^{-1}&=\Id_{\mcA_{L}}, \\
\label{commutative}\mu\circ c_{\mcA_{L},\mcA_{L}} &= \mu,
\end{align}
where $a_{\mcA_{L},\mcA_{L},\mcA_{L}}$ is the associativity constraint, $l_{\mcA_L}$ the left unit constraint, and $c_{\mcA_{L},\mcA_{L}}$ the braiding. The unit object $\mbbC_0$ is a summand of $\mcA_L$ so we take $\iota$ to be the inclusion map. Since $\mathrm{dimEnd}(\mbbC_{\lambda})=1$, we have 
\begin{equation}\label{product} \mu|_{\mbbC_{\lambda_1} \otimes \mbbC_{\lambda_2}}=t_{\lambda_1,\lambda_2} \mathrm{Id}_{\mbbC_{\lambda_1+\lambda_2}}\end{equation}
for some $t_{\lambda_1,\lambda_2} \in \mbbC$ and it is easy to see that $\mcA_L$ is a commutative algebra object if and only if these scalars satisfy
\begin{equation}
 t_{\lambda_1+\lambda_2,\lambda_3}t_{\lambda_1,\lambda_2}=t_{\lambda_1,\lambda_2+\lambda_3}t_{\lambda_2,\lambda_3},\qquad
\label{tcondition} t_{\lambda,0}=t_{0,\lambda}=1 \quad \text{and} \quad
t_{\lambda_1,\lambda_2}=t_{\lambda_2,\lambda_1}q^{\langle \lambda_1, \lambda_2 \rangle}.
\end{equation}

Let $\mcA'_L$ denote the same object with algebra structure given by structure constants $t'_{\lambda_1,\lambda_2}$. It is easy to see that $\mcA_L$ and $\mcA_L'$ are isomorphic if and only if there exist scalars $\phi_{\lambda}, \lambda \in L$ such that
\begin{equation}\label{algebrascalars} t'_{\lambda_1,\lambda_2}=\frac{\phi_{\lambda_1+\lambda_2}}{\phi_{\lambda_1}\phi_{\lambda_2}} t_{\lambda_1,\lambda_2}.\end{equation}
Denote by $\{ \gamma_i\}_{i=1}^m$ the generators of $L$. Then we can write any $\lambda \in L$ as $\lambda=\sum\limits_{i=1}^m n_i\gamma_i$ for some $n_i \in \mbbZ$, and we fix the notation
\begin{equation}\label{not1}
\lambda^{>k}=\sum\limits_{i>k}n_i\gamma_i, \qquad \lambda^{\geq k}= \sum\limits_{i \geq k} n_i\gamma_i, \qquad \lambda^k=n_k\gamma_k. 
\end{equation}
with $\lambda^{<k},\lambda^{\leq k}$ definied similarly. Fix non-zero scalars $\phi_{\gamma_i}$ and define the scalar $\phi_{\lambda}$ for any $\lambda \in L$ recursively by the relations
\begin{align*}
\phi_0&=1 & \phi_{n\gamma_i}&=\frac{\phi_{(n-1)\gamma_i}\phi_{\gamma_i}}{t_{(n-1)\gamma_i,\gamma_i}}  & \phi_{\lambda}=\frac{\phi_{\lambda^{<m}} \phi_{\lambda^m}}{t_{\lambda^{<m}, \lambda^m}}
\end{align*}
and let $t'_{\lambda_1,\lambda_2}$ be the scalars determined by equation \eqref{algebrascalars}. We have the following:
\begin{align}\label{1st}
t'_{0,\lambda}&=t'_{\lambda,0}=1, & t'_{n\gamma_i,\gamma_i}&=\frac{\phi_{(n+1)\gamma_i}}{\phi_{n\gamma_i}\phi_{\gamma_i}}t_{n\gamma_i,\gamma_i}=1, & t'_{\lambda^{<k},n\gamma_{k'}}&=\frac{\phi_{\lambda^{<k}+n\gamma_{k'}}}{\phi_{\lambda^{<k}}\phi_{n\gamma_{k'}}}t_{\lambda^{<k},n\gamma_{k'}}=1,
\end{align}
for all $n \in \mbbZ, i=1,...m, k' \geq k$. By associativity, we then have
\[ t'_{n\gamma_i,m\gamma_i}=\frac{t'_{(n+m-1)\gamma_i,\gamma_i}t'_{n\gamma_i,(m-1)\gamma_i}}{t'_{(m-1)\gamma_i,\gamma_i}}=t'_{n\gamma_i,(m-1)\gamma_i}.\]
We therefore see that $t'_{n\gamma_i,m\gamma_i}=1$ for all $n,m \in \mbbZ$ and $i=1,...,m$. We can translate this and equation \eqref{1st} into the notation of \eqref{not1} as  $t_{\lambda^k,\mu^k}=1$ and  $t_{\lambda^{< k},\mu^k}=1$
for all $\lambda,\mu \in L$ and $k=1,...,m$. We then see that
\begin{equation}\label{2nd} t'_{\lambda^{\leq k},\mu^k}=\frac{t'_{\lambda^k,\mu^k}t'_{\lambda^{<k},\lambda^k+\mu^k}}{t'_{\lambda^{<k},\lambda^k}}=1\end{equation}
for any $\lambda,\mu \in L$. It follows again by associativity and equation \eqref{2nd} that for any $k$, $t'_{\lambda,\mu^{<k}}=t'_{\lambda+\mu^{<k-1},\mu^{k-1}}t'_{\lambda,\mu^{<k-1}}$. Then,
\begin{equation}
t'_{\lambda,\mu}=t'_{\lambda+\mu^{<m},\mu^m}t'_{\lambda,\mu^{<m}}
\label{3rd}=t'_{\lambda+\mu^{<m},\mu^m}t'_{\lambda+\mu^{<m-1},\mu^{m-1}}t'_{\lambda,\mu^{<m-2}}
=\cdots =\prod\limits_{k=1}^mt'_{\lambda+\mu^{<k},\mu^k}
\end{equation}
Then, for any $\lambda,\mu \in L$ and any $k=1,...,m$ we have
\begin{equation*}
t'_{\lambda^{\leq k},\mu^{>k}}=\prod\limits_{s=1}^m t'_{\lambda^{\leq k}+(\mu^{>k})^{<s},(\mu^{>k})^s}
=\prod\limits_{s=k+1}^mt'_{\lambda^{\leq k}+(\mu^{>k})^{<s}),\mu^s}
=\prod\limits_{s=k+1}^m t'_{(\lambda^{\leq k}+ \mu^{>k})^{<s},\mu^s}=1
\end{equation*}
 where we have used equation \eqref{1st} and \eqref{2nd}, and the fact that $s>k$ in the last line. So, we have $t'_{\lambda,\mu^k}=t'_{\lambda^{\leq k}+\lambda^{>k},\mu^k}t'_{\lambda^{\leq k},\lambda^{>k}}=t'_{\lambda^{\leq k},\lambda^{>k}+\mu^k}t'_{\lambda^{>k},\mu^k}=t'_{\lambda^{>k},\mu^k}$ for any $\lambda,\mu \in L$. In particular, $t'_{\lambda+\mu^{<k},\mu^k}=t'_{\lambda^{>k},\mu^k}$, so equation \eqref{3rd} can be rewritten as
\begin{align}\label{decomp}
t'_{\lambda,\mu}=\prod\limits_{k=1}^mt'_{\lambda^{>k},\mu^k}=\prod\limits_{k=1}^mq^{\langle \lambda^{>k},\mu^k \rangle} t'_{\mu^k,\lambda^{>k}}=\prod\limits_{k=1}^m q^{\langle \lambda^{>k},\mu^k \rangle}
\end{align}

We are now ready to evaluate equations \ref{tcondition}. Let $\lambda_i=\sum\limits_{j=1}^m n_j^i \gamma_j \in L$ ($i=1,2,3$). Then, we have
\begin{align*}
\qquad & & t'_{\lambda_1+\lambda_2,\lambda_3}t'_{\lambda_1,\lambda_2}&=t'_{\lambda_1,\lambda_2+\lambda_3}t'_{\lambda_2,\lambda_3}\\
\Leftrightarrow & & \left( \prod\limits_{k=1}^m t'_{\lambda_1+\lambda_2,\lambda_3^k} \right) \left( \prod\limits_{k=1}^m t'_{\lambda_1, \lambda_2^k} \right)&= \left( \prod\limits_{k=1}^m t'_{\lambda_1, \lambda_2^k + \lambda_3^k} \right) \left( \prod\limits_{k=1}^mt'_{\lambda_2, \lambda_3^k} \right)\\
\Leftrightarrow & & \left( \prod\limits_{k=1}^m q^{\langle \lambda_1^{>k} + \lambda_2^{>k}, \lambda_3^k \rangle} \right) \left( \prod\limits_{k=1}^m q^{\langle \lambda_1^{>k}, \lambda_2^k \rangle} \right) &= \left( \prod\limits_{k=1}^m q^{\langle \lambda_1^{>k}, \lambda_2^k + \lambda_3^k \rangle} \right) \left(\prod\limits_{k=1}^m q^{ \langle \lambda_2^{>k} , \lambda_3^k \rangle} \right)\\
\Leftrightarrow & & \prod\limits_{k=1}^m q^{\langle \lambda_1^{>k},\lambda_3^k \rangle }q^{\langle \lambda_2^{>k}, \lambda_3^k \rangle} q^{\langle \lambda_1^{>k}, \lambda_2^k \rangle} &= \prod\limits_{k=1}^m q^{\langle \lambda_1^{>k},\lambda_2^k \rangle}q^{\langle \lambda_1^{>k}, \lambda_3^k \rangle} q^{\langle \lambda_2^{>k}, \lambda_3^k \rangle}
\end{align*}
which is trivially true for all $\lambda_i \in L$. Therefore, $\mcA_L$ is an algebra object for all $L \subset \mcL$. For commutativity, we obtain the following:
\begin{align*}
\qquad & & t'_{\lambda_1,\lambda_2}&=t'_{\lambda_2,\lambda_1}q^{\langle \lambda_1,\lambda_2 \rangle}\\
\Leftrightarrow & & \prod\limits_{k=1}^m t'_{\lambda_1, \lambda_2^k}&=q^{ \langle \lambda_1, \lambda_2 \rangle} \prod\limits_{k=1}^m t'_{\lambda_2, \lambda_1^k}\\
\Leftrightarrow& & \prod\limits_{k=1}^m q^{\langle \lambda_1^{>k}, \lambda_2^k \rangle}&=q^{\langle \lambda_1, \lambda_2 \rangle} \prod\limits_{k=1}^m q^{\langle \lambda_2^{>k}, \lambda_1^k \rangle}\\
\Leftrightarrow & & q^{\sum\limits_{k=1}^m \sum\limits_{j>k} n_j^1n_k^2 \langle \gamma_j, \gamma_k \rangle}&=q^{\sum\limits_{k=1}^m\sum\limits_{j>k} (n_k^1n_j^2+n_j^1n_k^2)\langle \gamma_j , \gamma_k \rangle + \sum\limits_{k=1}^m n_k^1n_k^2 \langle \gamma_k, \gamma_k \rangle} q^{\sum\limits_{k=1}^m \sum\limits_{j>k} n_j^2n_k^1 \langle \gamma_j, \gamma_k \rangle}\\
\Leftrightarrow & & 1&=q^{\sum\limits_{k=1}^m n_k^1n_k^2 \langle \gamma_k, \gamma_k \rangle}q^{\sum\limits_{k=1}^m \sum\limits_{j>k} 2n_k^1n_j^2\langle \gamma_j,\gamma_k \rangle}.\\
\end{align*}
Clearly, this holds for all $n_k^1,n_j^2 \in \mbbZ$ if $\langle \gamma_k, \gamma_k \rangle \in \ell \mbbZ$ for all $k$, and $2 \langle \gamma_j, \gamma_k \rangle \in \ell \mbbZ$ whenever $j \not = k$. It is also easy to see that we can choose appropriate coefficients to obtain the equations $q^{\langle \gamma_k , \gamma_k \rangle}=1$ and $q^{2 \langle \gamma_j, \gamma_k \rangle}=1$. Hence, commutativity holds if and only if 
$\langle \gamma_i , \gamma_i \rangle \in \ell \mbbZ$ and 
$2\langle \gamma_i, \gamma_j \rangle \in \ell \mbbZ\ \mathrm{for} \;  i\not =j$.
It is easy to check that this conditions implies 
$\langle \lambda,\lambda \rangle \in \ell \mbbZ$ and
$2\langle \lambda,\mu \rangle \in \ell \mbbZ$
for all $\lambda,\mu \in L$.

\end{proof}
We have the following useful corollary of the proof of Theorem \ref{algebraobject}:
\begin{corollary}\label{scalar}
The scalars $t_{\lambda_1,\lambda_2}$ of equation \eqref{product} defining the product $\mu:\mcA_L \otimes \mcA_L \to \mcA_L$ are non-zero.
\end{corollary}

Let $L^{\mu}$ be the lattice defined by adding a generator $\mu \in \mcL$ to a lattice $L \subset \mcL$, and let 
\[ \mcA_{L^{\mu}}=\bigoplus\limits_{\lambda \in L^{\mu}} \mbbC_{\lambda} \]
 be the corresponding algebra object. Note that if $\mu=0$, we recover the usual algebra objects defined above. If $\mu \not \in L$ and $2\mu \in L$, then every element in $L^{\mu}$ is in $L$ or $\mu+L$, so we can decompose $\mcA_{L^{\mu}}$ as $\mcA_{L^{\mu}}^{\bar{0}} \oplus \mcA_{L^{\mu}}^{\bar{1}}$ where
\[ \mcA_{L^{\mu}}^{\bar{0}}=\bigoplus\limits_{\lambda \in L} \mbbC_{\lambda}=\mcA_L,   \qquad    \mcA_{L^{\mu}}^{\bar{1}}=\bigoplus\limits_{\lambda \in L} \mbbC_{\mu+\lambda}, \]
and it is easy to show that $\mu(\mcA_{L^{\mu}}^{\bar{i}} \otimes \mcA_{L^{\mu}}^{\bar{j}}) \subset \mcA_{L^{\mu}}^{\overline{i+j}}$ since $\mathrm{dim}\mathrm{Hom}(\mbbC_{\lambda_1},\mbbC_{\lambda_2})=\delta_{\lambda_1,\lambda_2}$. Hence, $\mcA_{L^{\mu}}$ is a superalgebra object, and we have shown the first part of the following proposition:
\begin{proposition}\label{superalg}
Let $L \subset \mcL$ such that $\mcA_L$ is commutative and $\mu \in \mcL$ such that $\mu \not \in L$ and $2 \mu \in L$, then $\mcA_{L^{\mu}}$ is a superalgebra. $\mcA_{L^{\mu}}$ is supercommutative if and only if
\begin{align*}
2\langle \mu, \mu \rangle \in \ell \mbbZ \setminus 2\ell \mbbZ  \qquad \text{and} \qquad  2\langle \mu, \lambda \rangle  \in \ell \mbbZ 
\end{align*}
for all $\lambda\in L$.
\end{proposition}

\begin{proof}
It follows from the proof of Theorem \ref{algebraobject} that $\mcA_{L^{\mu}}$ is isomorphic to an algebra object with structure constants $t_{\lambda,\mu}$ as in equation \eqref{product} satisfying equation \eqref{decomp}. Recall that
\[ \mcA_{L^{\mu}}=\bigoplus\limits_{\lambda \in L} \mbbC_{\lambda} \oplus \bigoplus\limits_{\lambda \in L}\mbbC_{\mu+\lambda} \]
where $\bigoplus\limits_{\lambda \in L} \mbbC_{\lambda}$ and $\bigoplus\limits_{\lambda \in L}\mbbC_{\mu+\lambda}$ are the $\bar{0}$ and $\bar{1}$ components of the $\mbbZ_2$-grading respectively. For $\mcA_{L^{\mu}}$ to be supercommutative, $\mcA_L$ must be commutative which holds if and only if $\langle \lambda,\lambda \rangle \in \ell \mbbZ$ and $2 \langle \lambda, \lambda' \rangle \in \ell \mbbZ$ by Theorem \ref{algebraobject}. It then follows from equation \eqref{supercom} that to show $\mcA_{L^{\mu}}$ is supercommutative, we need to show that 
\begin{align}\label{scomm1}
t_{\mu+\lambda_1,\mu+\lambda_2}&=-q^{\langle \mu+\lambda_1, \mu+\lambda_2 \rangle} t_{\mu+\lambda_2,\mu+\lambda_1}\\
\label{scomm2} t_{\mu+\lambda_1,\lambda_2}&=q^{\langle \mu+\lambda_1,\lambda_2 \rangle}t_{\lambda_2,\mu+\lambda_1}\\
\label{scomm3} t_{\lambda_1,\mu+\lambda_2}&=q^{\langle \lambda_1, \mu+\lambda_2 \rangle}t_{\mu+\lambda_2,\lambda_1}
\end{align}
for all $\lambda_1,\lambda_2 \in L$. Let $\gamma_1,...,\gamma_{m}$ be a generating set for $L$, and $\gamma_{m+1}=\mu$. We can apply Equation \eqref{decomp} to obtain the following relations
\begin{align*}
t_{\mu+\lambda_1,\mu+\lambda_2}&=q^{\langle \mu, \lambda_2 \rangle} t_{\lambda_1,\lambda_2},\qquad
t_{\mu+\lambda_1,\lambda_2}&=q^{\langle \mu, \lambda_2 \rangle}t_{\lambda_1,\lambda_2},\qquad
t_{\lambda_1, \mu+\lambda_2}&=t_{\lambda_1,\lambda_2},
\end{align*}
for all $\lambda_1,\lambda_2 \in L$. It then follows from an easy computation using the fact that $t_{\lambda_1,\lambda_2}=q^{\langle \lambda_1, \lambda_2 \rangle}t_{\lambda_2,\lambda_1}$ (from the proof of Theorem \ref{algebraobject}) that $\mcA_{L^{\mu}}$ is supercommutative if and only if
$2\langle \mu, \mu \rangle \in \ell \mbbZ \setminus 2\ell \mbbZ$  and $2\langle \mu, \lambda \rangle  \in \ell \mbbZ$ 
for all $\lambda \in L$. 
\end{proof}

The categories $\mcC$ and $\mcC \boxtimes \mcH$ consist of finite length modules and therefore all indecomposables have simple submodules. If we assume this property for $\mathcal{B}$, then we have the following proposition which was proven in \cite[Proposition 4.4]{CKM} for module categories of vertex operator algebras, but holds more generally:

\begin{proposition}\label{simpleinduction} Suppose every indecomposable object in $\mcB$ has a simple subobject. Then $N \in \mathrm{Rep}^0\mcA_{L^{\mu}}$ is simple if and only if $N \cong \mcF(M)$ for a simple object $M \in \mcB$.
\end{proposition}
\begin{proof}
Let $M \in \mcB$ be simple and suppose $\mcF(M) \in \mathrm{Rep}\,\mcA_{L^{\mu}}$ is reducible with proper subobject $X$. We then have a $\mathrm{Rep}^0\mcA_{L^{\mu}}$-morphism $X \hookrightarrow \mcF(M)$ which gives a $\mcB$-morphism 
\[ \mcG(X) \hookrightarrow \mcG(\mcF(M))=\bigoplus\limits_{\lambda \in L^{\mu}} M^{\lambda}\in \mcB^{\oplus}\] by Frobenius reciprocity, where we have adopted the notation $M^{\lambda}:=\mbbC_{\lambda} \otimes M$. Since every $M^{\lambda} \in \mcB$ is simple, we have that 
\[\mcG(X)=\bigoplus\limits_{\lambda \in T} M^{\lambda}\]
for some non-empty subset $T \subset L^{\mu}$. By Corollary \ref{scalar}, the scalars $t_{\lambda_1,\lambda_2}$ defining the product on $\mcA_{L^{\mu}}$ are non-zero, so we see that the action of $\mcA_{L^{\mu}}$ on $\mcF(M)$ satisfies
\[ \mbbC_{\lambda} \cdot M^{\gamma} \cong M^{\lambda+\gamma}. \]
In particular, any summand $M^{\gamma}$ generates $\mcF(M)$, so $X=\mcF(M)$ and we see that $\mcF(M)$ has no non-trivial proper subobjects. It follows that the induction $\mcF(M)$ of a simple object $M \in \mcB$ is always simple.

Let $N \in \mathrm{Rep}^0\mcA_{L^{\mu}}$ be simple and $N'$ an indecomposable summand in $\mcG(N) \in \mcB^{\oplus}$ (so $N' \in \mcB$). Then $N'$ contains a simple submodule $M$ and we therefore have a non-zero map $f:M \to N' \hookrightarrow \mcG(N)$. By Frobenius reciprocity, we obtain a non-zero $\mathrm{Rep}\,\mcA_{L^{\mu}}$-morphism $g:\mcF(M) \to N$, which is an isomorphism because $N$ and $\mcF(M)$ are both simple.
\end{proof}
Given any superalgebra object $A=A^{\bar{0}} \oplus A^{\bar{1}}$, it is easy to check that $A^{\bar{1}}\in \mathrm{Rep}\,A^{\bar{0}}$ with action given by the restriction $\mu|_{A^{\bar{0}} \otimes A^{\bar{1}}}:A^{\bar{0}} \otimes A^{\bar{1}} \to A^{\bar{1}}$ of the product $\mu:A \otimes A \to A$. Suppose $A$ is supercommutative and a direct sum of simple currents such that $A^{\bar{0}}=\mcA_L$ for some lattice $L \subset \mcL$ and $A^{\bar{1}}$ is a non-trivial simple object in $\mathrm{Rep}\,A^{\bar{0}}$. Since $A$ is supercommutative, $A^{\bar{0}}=\mcA_L$ is commutative and because $A^{\bar{1}}$ is simple in $\mathrm{Rep}\,A^{\bar{0}}$, we know $A^{\bar{1}} \cong \mcF(M)$ for some simple module $M \in \mcB$ by Proposition \ref{simpleinduction}. Since $A$ is a direct sum of simple currents, $A^{\bar{1}}\cong\mcF(M)=\bigoplus\limits_{\lambda \in L} \mbbC_{\lambda} \otimes M$ is a direct sum of simple currents, so we must have $M \cong \mbbC_{\mu}$ for some $\mu \in \mcL$. Hence, we have
\[ A= A^{\bar{0}} \oplus A^{\bar{1}} \cong \mcA_L \oplus \mcG(\mcF(\mbbC_{\mu})) = \bigoplus\limits_{\lambda \in L} \mbbC_{\lambda} \oplus \bigoplus\limits_{\lambda \in L}\mbbC_{\mu+\lambda} =\mcA_{L^{\mu}}.\]
Since $m(\mbbC_{\mu} \otimes \mbbC_{\mu}) \subset \mcA_L$ we have $2\mu \in L$, and $\mu \not \in L$ otherwise $A^{\bar{1}}=\mcG(\mcF(\mbbC_{\mu}))=\mcA_L$ is trivial in $\mathrm{Rep}\mcA_L$. We therefore have the following:
\begin{corollary}\label{superalgclass}
Let $A=A^{\bar{0}} \oplus A^{\bar{1}}$ be a supercommutative superalgebra which is a direct sum of simple currents such that $A^{\bar{0}}=\mcA_L$ for some lattice $L \subset \mcL$ and $A^{\bar{1}}$ is a non-trivial simple object in $\mathrm{Rep}\,A^{\bar{0}}$. Then $A=\mcA_{L^{\mu}}$ for some lattice $L^{\mu} \subset \mcL$ is of the type described in Proposition \ref{superalg}.
\end{corollary}

\section{Examples}
Recall from Definition \ref{weightmod} that a $\U$-module $V$ is called a weight module if the $H_i$ act semi-simply on $V$ and we have $K_{\gamma}=q^{\sum\limits_{i=1}^n k_id_i H_i}$ as operators on $V$ for $\gamma = \sum\limits_{i=1}^n k_i \alpha_i \in \mathsf{L}$, and $\mcC$ denotes the category of weight modules for $\U$ at $\ell$-th root of unity. We denote by $M^{\lambda}$ and $\irred{\lambda}$ the Verma and irreducible modules of highest weight $\lambda \in \mfh^*$ in $\mcC$, and by $P^{\lambda}$ the projective cover of $\irred{\lambda}$. As noted in \cite[Remark 4.6]{R}, the simple-currents in $\C$ are given by the set
\begin{equation}\label{currents} \left\{ S^{\lambda} \, | \, \lambda \in \mcL \right\} \end{equation}
where $\mcL:= \{ \lambda \in \mfh^* \, | \,  \lambda(H_i) \in \frac{\ell}{2d_i}\mbbZ\}$. Throughout this section, we will denote by $\Gamma(X)$ the set of weights of a module $X \in \C$, and we adopt the notation $\mbbC_{\lambda}:=S^{\lambda}$ when $\lambda \in \mcL$.

\begin{lemma}\label{ssbraid}
For any $\lambda \in \mcL$, the braiding $c_{-,-}$ acts as $\tau \circ \mathscr{H}$ (recall Equation \ref{braidingH}) on $X \otimes \mbbC_{\lambda}, \mbbC_{\lambda} \otimes X \in \C$ for all $X \in \C$.
\end{lemma}
\begin{proof}
Recall that the braiding in $\C$ is given by $\tau \circ \mathscr{H}\mathscr{R}$ where $\tau$ is the flip map and $\mathscr{H}$ and $\mathscr{R}$ are defined by Equation \ref{braidingH} and \ref{braidingR}. It follows immediately from equation \eqref{braidingR} that action of the braiding coincides with that of $\tau \circ \mathscr{H}$ on $X \otimes \mbbC_{\lambda}$ and $\mbbC_{\lambda} \otimes X \in \C$ since for all $i$, $X_{\pm i}$ acts as zero on $\mbbC_{\lambda}$.
\end{proof}
\begin{theorem}\label{induction}${}$
 \begin{itemize}
\item $\mcF(P^{\lambda}) \in \mathrm{Rep}^0(\mcA_{L^{\mu}})$ if and only if $\lambda \in \frac{\ell}{2}(L^{\mu})^*$.
\item Let $X \in \C$ and let $P_X \in C$ be the projective cover of $X$. Then $\mcF(X) \in \mathrm{Rep}^0(\mcA_{L^{\mu}})$ if and only if $\mcF(P_X) \in \mathrm{Rep}^0\mcA_{L^{\mu}}$.
\item $\mcF(P^{\lambda})$ is the projective cover of $\mcF(\irred{\lambda})$ in $\mathrm{Rep}^0\mcA_{L^{\mu}}$.
\item The distinct irreducible objects in $\mathrm{Rep}^0\mcA_{L^{\mu}}$ are given by the set $\{ \mcF(\irred{\lambda}) \, | \, \lambda \in \Lambda(L^{\mu})\}$ where $\Lambda(L^{\mu}):=\frac{\ell}{2}(L^{\mu})^*/\frac{\ell}{2}(L^{\mu})^* \cap L^{\mu}$. $\mathrm{Rep}^0\mcA_{L^{\mu}}$ is finite if and only if $\mathrm{rank}(L^{\mu})=\mathrm{rank}(P)$.
\end{itemize}
\end{theorem}
\begin{proof}
Recall from Definition \ref{local} that $\mcF(X) \in \mathrm{Rep}^0(\mcA_{L^{\mu}})$ if and only if $M_{\mbbC_{\lambda}, X}=c_{X,\mbbC_{\lambda}} \circ c_{\mbbC_{\lambda}, X} =\mathrm{Id}$ for all $\lambda \in L^{\mu}$. It follows from Lemma \ref{ssbraid} and equation \eqref{braidingH} that for any vector $w_{\gamma} \in X$ of weight $\gamma \in \mfh^*$ we have
\begin{equation}\label{eq1} M_{\mbbC_{\lambda},X}(v_{\lambda} \otimes w_{\gamma})=q^{2\langle \lambda, \gamma \rangle} \mathrm{Id}\end{equation}
We therefore have $X$ local if and only if $2\langle \lambda, \gamma \rangle \in \ell \mbbZ$ for all $\gamma \in \Gamma(X)$ and $\lambda \in L^{\mu}$, which holds if and only if $\gamma \in \frac{\ell}{2}(L^{\mu})^*$ for all $\gamma \in \Gamma(X)$. Recall, however, that $L^{\mu} \subset \mcL$, so it is easy to see that $\alpha \in \frac{\ell}{2} (L^{\mu})^*$ for all $\alpha \in Q$. For any cyclic indecomposable $X \in \C$, all of the weights of $X$ differ by an element of $Q$, so $\mcF(X)\in \mathrm{Rep}^0\mcA_{L^{\mu}}$ if and only if any one of its weights $\gamma \in \Gamma(X)$ satisfies equation \eqref{eq1}. In particular, $\mcF(P^{\lambda}) \in \mathrm{Rep}^0\mcA_{L^{\mu}}$ if and only if $\lambda \in \frac{\ell}{2} (L^{\mu})^*$ and $\mcF(X) \in \mathrm{Rep}^0\mcA_{L^{\mu}}$ if and only if $\mcF(P_X) \in \mathrm{Rep}^0\mcA_{L^{\mu}}$ as their weights differ by elements in $Q$.

Define $G:=\{ \gamma \in \mfh^*| \; 2\langle \lambda, \gamma \rangle \in \ell \mbbZ \quad \forall \lambda \in L^{\mu} \}=\frac{\ell}{2}(L^{\mu})^*$, so $\mcF(P^{\gamma}) \in \mathrm{Rep}^0(\mcA_{L^{\mu}})$ if and only if $\gamma \in G$. It is easy to see that two irreducible modules  $\irred{\gamma_1},\irred{\gamma_2}$ induce to the same module if and only if $\gamma_1 - \gamma_2 \in L^{\mu}$. In particular, the set of distinct $\mcF(L^{\gamma})$ is in bijective correspondence with the quotient of free abelian groups $G/(G \cap L^{\mu})$, which is finite if and only if $\mathrm{rank}(G)=\mathrm{rank(G \cap L^{\mu})}$, which holds if and only if $L^{\mu}$ has full rank ($\mathrm{rank}(L^{\mu})=\mathrm{rank}(P)$).

It follows immediately from Frobenius reciprocity that $\mcF(P^{\lambda})\in \mathrm{Rep}^0\mcA_{L^{\mu}}$ is projective, and the surjection $f:P^{\lambda} \twoheadrightarrow \irred{\lambda}$ induces to a surjection $\mcF(f)=Id_{\mcA_L} \otimes f: \mcF(P^{\lambda}) \to \mcF(\irred{\lambda})$. If $\mcF(P^{\lambda})$ is not the projective cover of $\mcF(\irred{\lambda})$, then $\mcF(f)$ must not be essential. That is, there exists a proper submodule $A \subset \mcF(P^{\lambda})$ such that $\mcF(f)(A)=\mcF(\irred{\lambda})$. Note that 
\[ \mcF(P^{\lambda})\cong \bigoplus\limits_{\gamma \in L} P^{\lambda+\gamma} \qquad \qquad \mcF(\irred{\lambda})\cong \bigoplus\limits_{\gamma \in L} \irred{\lambda+\gamma} \]
and we must have $\mcF(f)|_{P^{\lambda+\gamma}} :P^{\lambda+\gamma} \to \irred{\lambda+\gamma}$. Therefore, if $f$ is not essential, there exists a $P^{\gamma}$ with a submodule $A'$ which surjects onto $\irred{\gamma}$, contradicting the fact that $P^{\gamma}$ is the projective cover of $\irred{\gamma}$.
\end{proof}

Recall that $r=\ell$ if $\ell$ is odd and $r=\ell/2$ if $\ell$ is even. We have the following:
\begin{lemma}\label{ext}
If $r \nmid 2d_i$ for all $i$, then $\mathrm{Ext}^1(\mbbC_{\lambda_1},\mbbC_{\lambda_2})=0$ for all $\lambda_1,\lambda_2 \in \mcL$.
\end{lemma}
\begin{proof}
Let $N \in \mathrm{Ext}^1(\mbbC_{\lambda_1},\mbbC_{\lambda_2})$, that is,
\begin{equation}\label{ses} 0 \to \mbbC_{\lambda_2} \to N \to \mbbC_{\lambda_1}\to 0. \end{equation}
If $\lambda_1=\lambda_2$, then the sequence splits since $X_{\pm i}$ acts as zero for all $i$ as $\mathrm{ch}[N]=2z^{\lambda}$ and the $H_i$ act semisimply. Suppose now that $\lambda_1 > \lambda_2$. If the sequence does not split, we must have $\lambda_1=\lambda_2-\alpha_i$ for some $i$ (otherwise, $X_{\pm i}v_{\lambda_1}=0$ for all i and the sequence splits as above). Any vector $n_{\lambda_1} \in N(\lambda_1)$ is highest weight and generates $N$. We therefore have a quotient map $\phi:M^{\lambda_1} \twoheadrightarrow N$ $v_{\lambda_1} \mapsto n_{\lambda_1}$, from the Verma module $M^{\lambda_1}$ of highest weight $\lambda_1$ onto $N$. However, $\lambda_1 \in \mcL$, so it follows from equation \eqref{eqX} that we have
\begin{align*} X_iX_{-i}^2v_{\lambda_1}&=[\lambda_1(H_i)-2]_iv_{\lambda_1}=-[2]_iX_{-i}v_{\lambda_1}.
\end{align*}
$[2]_i=0$ if and only if $r|2d_i$. We then see that $X_{-i}^2v_{\lambda_1}\not = 0$, a contradiction, so the module $N$ cannot exist if $r\nmid 2d_i$ for all $i$. If $\lambda_1 < \lambda_2$, then the above argument tells us that the dualized sequence
\[ 0 \to \mbbC_{\lambda_1} \to \check{N} \to \mbbC_{\lambda_2} \to 0\]
splits, where $\check{(-)}$ is the exact contravariant functor defined in \cite[Subsection 4.2]{R}, hence the original sequence splits. 
\end{proof}

Let $L^{\mu} \subset \mcL$ such that $\mcA_{L^{\mu}}$ is a supercommutative superalgebra object. Recall from Equation \eqref{twist}, that the twist acts on $\mathbb{C}_{\gamma}$ as
\begin{align*}
\theta_{\mbbC_{\gamma}}&=q^{\langle \gamma , \gamma+2(1-r)\rho\rangle} \text{Id}
 =q^{\langle \gamma, \gamma \rangle + 2(1-r)\langle \gamma, \rho\rangle}\text{Id}.
\end{align*}

It then follows from \cite[Proposition 2.86]{CKM} and Proposition \ref{superalg} that $\mathrm{Rep}^0\mcA_{L^{\mu}}$ is ribbon if $2(1-r)\langle \lambda, \rho \rangle \in \ell \mbbZ$ for all $\lambda \in L$ and $2(1-r)\langle \mu,\rho \rangle \in \frac{\ell}{2}\mbbZ$. This gives us the first statement of the following proposition, where we recall that $\Lambda(L^{\mu}):=\frac{\ell}{2}(L^{\mu})^*/\frac{\ell}{2}(L^{\mu})^* \cap L^{\mu}$.
\begin{proposition}\label{rib}
Let $\mcA_{L^{\mu}}$ be a supercommutative superalgebra. Then
\begin{itemize}
\item $\mathrm{Rep}^0\mcA_{L^{\mu}}$ is ribbon if $2(1-r)\langle \lambda, \rho \rangle \in \ell \mathbb{Z}$ for all $\lambda \in L$ and $2(1-r)\langle \mu, \rho \rangle \in \frac{\ell}{2}\mbbZ$.
\item If $r\nmid 2d_i$ for all $i$, then $\mathrm{Rep}^0\mcA_{L^{\mu}}$ has non-trivial M\"{u}ger center if and only if there exists a $\lambda\in \Lambda(L^{\mu})$ such that $\langle \lambda, \gamma \rangle \in \frac{\ell}{2} \mbbZ$ for all $\gamma \in \Lambda(L^{\mu})$.
\end{itemize}
\end{proposition}
\begin{proof}
As the braiding on induced modules (hence all simple modules) in $\mathrm{Rep}^0\mcA_{L^{\mu}}$ descends from that of $\mcC^{\oplus}$ \cite[Theorem 2.67]{CKM} (i.e. acts by the same scalar), an irreducible module $\mcF(\irred{\lambda})\in \mathrm{Rep}^0\mcA_{L^{\mu}}$ is transparent to all other irreducible modules $\mcF(\irred{\lambda'}) \in \mathrm{Rep}^0\mcA_{L^{\mu}}$ if and only if it satisfies 
\[ c_{\mcF(\irred{\lambda}),\mcF(\irred{\lambda'})} \circ c_{\mcF(\irred{\lambda'}), \mcF(\irred{\lambda})}=q^{2\langle \lambda, \lambda' \rangle} \mathrm{Id}_{\mcF(\irred{\lambda}) \otimes \mcF(\irred{\lambda'})}=\mathrm{Id}_{\mcF(\irred{\lambda}) \otimes \mcF(\irred{\lambda'})} \]
for all $\mcF(\irred{\lambda'}) \in \mathrm{Rep}^0\mcA_{L^{\mu}}$. Recall from Theorem \ref{induction} that irreducibles in $\mathrm{Rep}^0\mcA_{L^{\mu}}$ are given by the set $\{ \mcF(\irred{\lambda}) \, | \, \lambda \in \Lambda(L^{\mu})\}$ where $\Lambda(L^{\mu})=\frac{\ell}{2}(L^{\mu})^*/\frac{\ell}{2}(L^{\mu})^* \cap L^{\mu}$. Hence, $\mcF(\irred{\lambda})$ is transparent if and only if $\langle \lambda, \lambda' \rangle \in \frac{\ell}{2} \mathbb{Z}$ for all $\lambda' \in \Lambda(L^{\mu})$. It follows that if there is no such $\lambda \in \Lambda(L^{\mu})$, then the only transparent irreducible object is the unit object $\mcA_{L^{\mu}}$. If an indecomposable module $X \in \mathrm{Rep}^0\mcA_{L^{\mu}}$ is transparent, then all of the irreducible factors in its composition series are. Therefore, all transparent objects in $\mathrm{Rep}^0 \mcA_{L^{\mu}}$ are extensions of $\mcA_{L^{\mu}}$. Suppose now that we have an object $N \in \mathrm{Ext}^1_{\mathrm{Rep}^0{\mcA_L}}(\mcA_{L^{\mu}},\mcA_{L^{\mu}})$, that is,
\[ 0 \to \mcA_{L^{\mu}} \to N \to \mcA_{L^{\mu}} \to 0.\]
This sequence splits if and only if $\mathrm{dim\ Hom}(\mcA_{L^{\mu}},N)=2$. The restriction functor $\mcG:\mathrm{Rep}\,\mcA_{L^{\mu}} \to \mcC$ then yields a corresponding sequence 
\[ 0 \to \bigoplus\limits_{\lambda \in L^{\mu}} \mbbC_{\lambda} \xrightarrow{\iota} \mcG(N) \xrightarrow{\pi} \bigoplus\limits_{\lambda \in L^{\mu}}\mbbC_{\lambda} \to 0 \]
in $\mcC^{\oplus}$. Let $\iota_{\gamma}:=\iota|_{\mbbC_{\gamma}}$ and $N^{\gamma}$ the indecomposable factor in $\mcG(N)$ containing $\mathrm{Im}(\iota_{\gamma})$. Then we have an exact sequence
\[ 0 \to \mbbC_{\gamma} \xrightarrow{\iota_{\gamma}} N^{\gamma} \xrightarrow{\pi|_{N^{\gamma}}} \bigoplus\limits_{\lambda \in T} \mbbC_{\lambda} \to 0\]
for some finite $T \subset \mcL$. It follows from Lemma \ref{ext} that this sequence splits for all $\gamma$, so $\mathrm{dim\ Hom}_{\mcC}(\mbbC_{\gamma}, \mcG(N)) \geq 2$. By Frobenius reciprocity, we have $\mathrm{dim\ Hom}(\mcA_{L^{\mu}}, X) \geq 2$, so the original sequence splits. It follows that all transparent objects in $\mathrm{Rep}^0\mcA_{L^{\mu}}$ are trivial (direct sums of $\mcA_{L^{\mu}}$).
\end{proof}

\subsection{Braided Tensor Categories for $W_Q(r)$}\label{subtriplet}
The triplet vertex operator algebra $W_Q(r)$ associated to the root lattice $Q$ of a finite dimensional complex simple Lie algebra $\mfg$ of ADE type can be written as an extension of the singlet as
\begin{equation*}
W_Q(r)=\bigoplus\limits_{\lambda \in \sqrt{r}Q} M_{\lambda} \in (\mathrm{Rep}_{\langle s \rangle} W^0_Q(r))^{\oplus}.
\end{equation*}
The $M_\lambda$ are conjecturally simple currents of the singlet algebra. 
The algebra object in $\mcC^{\oplus}$ corresponding to the triplet vertex operator algebra under the map $\phi:\mathrm{Rep}_{\langle s \rangle} W^0_Q(r) \to \mathrm{Rep}\overline{U}_q^H(\mfg)^{wt}$ of Equation \eqref{id} is
\[\mcA_{rQ}=\bigoplus\limits_{\lambda \in rQ} \mbbC_{\lambda}. \]
It follows from Theorem \ref{induction} that the irreducible modules in $\mathrm{Rep}^0\mcA_L$ are
\[ \{\mcF(\irred{\lambda}) \, | \, \lambda \in P/rQ \} \]
Recall that the order $|P/rQ|$ is equal to the determinant of the change of basis matrix $A_{rQ \leftarrow P}$, which is easily seen to be $rA$ where $A$ is the Cartan matrix of $\mfg$, so $\mathrm{Rep}^0\mcA_{rQ}$ has $\mathrm{Det}(A) \cdot r^{\mathrm{rank}(\mfg)}$ distinct irreducible objects. Further, we have $\langle \rho, \alpha \rangle \in \mbbZ$ for all $\alpha \in Q$, so it follows from Proposition \ref{rib} that $\mathrm{Rep}^0\mcA_{rQ}$ is ribbon. Suppose $\lambda \in \mathrm{Rep}^0\mcA_L$ is transparent, i.e. $\langle \lambda, \gamma \rangle \in r\mbbZ$ for all $\gamma \in P/rQ$. It follows that $\langle \lambda, \omega_k \rangle \in r\mbbZ$ for all fundamental weights $\omega_k \in P$, so $\lambda \in Q/rQ$ and we can write $\lambda=\sum\limits_{k=1}^n m_i \alpha_i$ where $\sum\limits_{k=1}^n m_i\in r\mbbZ$ and $0 \leq m_i <r$. Then, if $\langle \lambda, \rho \rangle \in r\mbbZ$, choose $j$ such that $n_j \not = 0$, then $\langle \lambda, \rho+\omega_j \rangle = \langle \lambda, \rho \rangle +m_j \not \in r\mbbZ$. Hence, $\mathrm{Rep}^0\mcA_{rQ}$ has trivial M\"{u}ger center by Proposition \ref{rib}.

\begin{proposition}\label{tripletao}
$\mathrm{Rep}^0\mcA_{rQ}$ is a finite non-degenerate ribbon category (i.e. Log-Modular) with $\mathrm{det}(A)\cdot r^{\mathrm{rank}(\mfg)}$ distinct irreducible modules, where $A$ is the Cartan matrix of $\mfg$.
\end{proposition}
We expect that $\mathrm{Rep}^0\mcA_{rQ}$ and the category $\mathrm{Rep}_{\langle s \rangle} W_Q(r)$ generated by irreducible $W_Q(r)$-modules are ribbon equivalent.

\subsection{Braided Tensor Categories for $B_Q(r)$}\label{subbp}

Let $\mathsf{H}$ denote the Heisenberg vertex operator algebra of Subsection \ref{Heis} whose rank coincides with that of $\mfg$, and $\mathcal{H}$ the category of modules on which the zero-mode $L_0$ of the Virasoro element acts semi-simply. This category is semisimple with irreducible modules $\mathsf{F}_{\lambda}$, $\lambda \in \mathfrak{h}^*$ and fusion rules
\[ \mathsf{F}_{\lambda} \otimes \mathsf{F}_{\mu} \cong \mathsf{F}_{\lambda+\mu}. \] 
We see then every object $\mathsf{F}_{\lambda} \in \mathcal{H}$ is a simple current with tensor inverse $\mathsf{F}_{-\lambda}$. Hence, any object $\mbbC_{\lambda} \boxtimes \mathsf{F}_{\gamma} \in \mcC \boxtimes \mcH$ in the Deligne product of $\mcC$ $\mathcal{H}$ with $\gamma \in \mathfrak{h}^*$ and $\lambda \in \mcL= \{ \lambda \in \mfh^* \, | \,  \lambda(H_i) \in \frac{\ell}{2d_i}\mbbZ\}$ is a simple current. We can therefore define a family $\mfB^a_L$ of objects in $(\mcC \boxtimes \mcH)^{\oplus}$ where $a \in \mbbC$ as:
\[ \mfB^a_L:= \bigoplus\limits_{\lambda \in L}  \mbbC_{\lambda} \boxtimes \mathsf{F}_{a \lambda} \]
It follows from Lemma \ref{ssbraid} and Equation \eqref{Hbraid} that the braiding in $ \mcC \boxtimes \mcH$ acts on simple currents as $c_{\mbbC_{\lambda}\boxtimes \mathsf{F}_{a\lambda} , \mbbC_{\mu}\boxtimes  \mathsf{F}_{a\mu} }=q^{(1+pa^2) \langle \lambda, \mu \rangle}Id$. We then have the following corollary of Theorem \ref{algebraobject}:

\begin{corollary}\label{Bpalg}
$\mfB^a_L \in  (\mcC \boxtimes \mcH)^{\oplus}$ is a commutative algebra object if and only if
\begin{align*}
(1+pa^2)\langle \lambda,\lambda' \rangle & \in p\mbbZ,\\
(1+pa^2) \langle \lambda, \lambda \rangle & \in 2p \mbbZ,
\end{align*}
for all $\lambda, \lambda' \in L$.
\end{corollary}

The higher rank $B_Q(r)$ vertex operator algebras of ADE type can be realized as (see \cite[Section 3]{C1})
\begin{equation}
B_Q(r)=\bigoplus\limits_{\lambda \in P} M_{\sqrt{r}\lambda} \otimes \mathsf{F}_{\sqrt{-r}\lambda} \in \mathrm{Rep}_{\langle s \rangle}\mcW^0(r)_{\mfg} \otimes \mcH
\end{equation}
where $M_{\sqrt{r}\lambda}$ is a simple current for $W^0_Q(r)$, which corresponds to $\mbbC_{\lambda}$ under the correspondence in equation \eqref{id}. $B_Q(r)$ therefore corresponds to the object
\[ \mfB^{a_r}_{rP}:= \bigoplus\limits_{\lambda \in P}  \mbbC_{r\lambda} \boxtimes \mathsf{F}_{\sqrt{-r}\lambda} =\bigoplus\limits_{\lambda \in rP}  \mbbC_{\lambda} \boxtimes \mathsf{F}_{a_r\lambda}\]
where $a_r=\sqrt{-1/r}$, which clearly satisfies the conditions of Corollary \ref{Bpalg} since $1+ra_r^2=0$. It follows from Equations \eqref{twist} and \eqref{Htwist} that the twist acts on $\mathbb{C}_{\lambda} \boxtimes \mathsf{F}_{a_r \lambda}$ as
\begin{align*} \theta_{\mathbb{C}_{\lambda} \boxtimes \mathsf{F}_{a_r \lambda}}=\theta_{\mbbC_{\lambda}} \boxtimes \theta_{\mathsf{F}_{a_r \lambda}}=q^{\langle \lambda, \lambda+2(1-r)\rho \rangle - \langle \lambda, \lambda \rangle}=q^{2(1-r)\langle \lambda, \rho \rangle}.\\
\end{align*}
Hence, by \cite[Proposition 2.86]{CKM} $\mathrm{Rep}^0\mfB^{a_r}_{rP}$ is ribbon if and only if $(1-r)\langle \lambda, \rho \rangle \in  \mbbZ$ for all $\lambda \in P$. $\rho \in \frac{1}{2}Q$, so this holds if $r$ is odd, or $\rho \in Q$. The irreducible modules in $\mathrm{Rep}^0\mfB_{rP}^{a_r}$ are $\{ \mcF(\irred{\mu} \boxtimes \mathsf{F}_{\gamma}) \, | \, \langle \lambda, \mu+ra_r\gamma \rangle \in \mbbZ \; \text{for all $\lambda \in P$}\}$ with relations $\mcF(\irred{\mu} \boxtimes \mathsf{F}_{\gamma}) \cong \mcF(\irred{\mu+\lambda} \boxtimes \mathsf{F}_{\gamma+a_r\lambda})$ for all $\lambda \in rP$. The monodromy (double braiding) acts on pairs of irreducible modules $\mcF(\irred{\mu} \boxtimes \mathsf{F}_{\gamma})$, $\mcF(\irred{\mu'} \boxtimes \mathsf{F}_{\gamma'})$ as $M_{\mcF(\irred{\mu} \boxtimes \mathsf{F}_{\gamma}),\mcF(\irred{\mu'} \boxtimes \mathsf{F}_{\gamma'})}=q^{2\langle \mu, \mu' \rangle+2r\langle \gamma,\gamma' \rangle}$, so $\mcF(\irred{\mu} \boxtimes \mathsf{F}_{\gamma})$ is transparent if and only if 
\begin{equation}\label{bptrans}
2\langle \mu, \mu' \rangle + 2r \langle \gamma, \gamma' \rangle \in 2r\mbbZ
\end{equation}
for all $\mu',\gamma' \in \mfh^*$ such that $\langle \lambda, \mu'+ra_r\gamma' \rangle \in \mbbZ$ for all $\lambda \in P$. It is easy to see that $\mcF(\irred{\mu} \boxtimes \mathsf{F}_{\gamma}) \in \mathrm{Rep}^0\mfB^{a_r}_{rP}$ implies $\mcF(\irred{\mu+\alpha} \boxtimes \mathsf{F}_{\gamma+\frac{1}{ra_r}\beta}) \in \mathrm{Rep}^0\mcA_{rP}$ for all $\alpha, \beta \in Q$. It then follows from Equation \eqref{bptrans} that $\mcF(\irred{\mu} \boxtimes \mathsf{F}_{\gamma})$ is transparent if and only if $\mu \in rP$ and $\gamma \in ra_rP$. Let $\mu=r\tilde{\mu}$ and $\gamma=ra_r\tilde{\gamma}$ where $\tilde{\mu},\tilde{\gamma} \in P$. Then,
\begin{align*}
& & 2\langle \mu, \mu' \rangle+2r \langle \gamma, \gamma' \rangle &\in 2r\mbbZ \\
\Leftrightarrow & & 2r \langle \tilde{\mu},\mu'\rangle+2r^2a_r \langle \tilde{\gamma} , \gamma' \rangle &\in 2r\mbbZ \\
\Leftrightarrow & & \langle \tilde{\mu},\mu' \rangle +\langle \tilde{\gamma},\mu' \rangle-\langle \tilde{\gamma},\mu' \rangle+ra_r \langle \tilde{\gamma} ,\gamma' \rangle &\in \mbbZ\\
\Leftrightarrow & & \langle \tilde{\mu}-\tilde{\gamma}, \mu' \rangle+\langle \tilde{\gamma}, \mu'+ra_r\gamma' \rangle &\in \mbbZ\\
\Leftrightarrow & & \langle \tilde{\mu}-\tilde{\gamma}, \mu' \rangle &\in \mbbZ\\
\end{align*}
where we have used the fact that $\langle \lambda, \mu'+ra_r \gamma' \rangle \in \mbbZ$ for all $\lambda \in P$ and $\gamma' \in P$. Hence, $\mcF(\irred{\mu} \boxtimes \mathsf{F}_{\gamma})$ is transparent if and only if $\langle \tilde{\mu} -\tilde{\gamma},\mu' \rangle \in \mbbZ$ for all $\mu'$ such that $\mcF(\irred{\mu'} \boxtimes \mathsf{F}_{\gamma'}) \in \mathrm{Rep}^0\mfB^{a_r}_{rP}$ for some $\gamma' \in \mfh^*$. This clearly never holds as we can always choose a $\mu'$ not in the root lattice $Q$. Hence, there are no transparent objects and the argument in Proposition \ref{rib} can be applied to show that the M\"{u}ger center is trivial. We therefore make the following conjecture for the full subcategory $\mathrm{Rep}_{\langle s \rangle}B_Q(r)$ of $B_Q(r)$-modules generated by irreducibles:
\begin{proposition}\label{bpao}
$\mathrm{Rep}^0\mfB_{rP}^{a_r}$ is non-degenerate, ribbon if r is odd or $\rho \in Q$, and the irreducible modules are
\[ \{ \mcF(\irred{\mu} \boxtimes \mathsf{F}_{\gamma}) \; | \,  \mu,\gamma \in \mfh^*, \, \mathrm{and} \, \mu+ra_r\gamma \in Q \; \} \] 
with relations $\mcF( \irred{\mu}\boxtimes \mathsf{F}_{\gamma}) \cong \mcF(\irred{\mu+\lambda} \boxtimes \mathsf{F}_{\gamma+a_r\lambda})$ for all $\lambda \in rP$, where $a_r=\sqrt{-1/r}$.
\end{proposition}
We expect that $\mathrm{Rep}^0\mfB_{rP}^{a_r}$ and the category $\mathrm{Rep}_{\langle s \rangle} B_Q(r)$ generated by irreducible $B_Q(r)$-modules are ribbon equivalent.

\end{document}